\documentclass[11pt]{article}

\usepackage{amsmath}\usepackage{amsfonts}\usepackage{amssymb}\usepackage{amsthm}
\usepackage{graphicx}
\usepackage{ifthen}
\usepackage{lscape}
\usepackage{subfigure}
\usepackage{setspace}
\usepackage{subeqnarray}
\usepackage{url}

\usepackage[bw]{mcode}

\newenvironment{keywords}{\footnotesize{\bf Keywords: }}{}
\newenvironment{AMS}{\footnotesize{\bf AMS subject classification: }}{}
\usepackage[left=1in,top=1.5in,right=1in,bottom=1in]{geometry}

\usepackage{verbatim}
\usepackage{mdwlist}
\usepackage{xcolor}

\newtheorem{theorem}{Theorem}[section]
\newtheorem{lemma}[theorem]{Lemma}

\newtheorem{proposition}[theorem]{Proposition}
\newtheorem{remark}{Remark}[section]
\numberwithin{equation}{section}
\numberwithin{figure}{section}

\numberwithin{equation}{section}
\numberwithin{figure}{section}

\long\def\symbolfootnote[#1]#2{\begingroup\def\thefootnote{\fnsymbol{footnote}}
\footnote[#1]{#2}\endgroup}

\renewcommand{\paragraph}[1]{}
\usepackage{verbatim}
\usepackage{ifpdf}
\ifpdf
\else
\renewcommand{\includegraphics}[1]{\framebox{Graphics Placeholder}}
\renewcommand{\includegraphics}[2][1]{\framebox{Graphics Placeholder}}
\fi

\author{Alexander V.\ Shapeev\thanks{Section of Mathematics, Swiss Federal Institute of Technology (EPFL), Station 8, CH-1015, Lausanne, Switzerland 
        ({\tt alexander@shapeev.com}).
        }\thanks{Present address: 
                School of Mathematics, 206 Church St. SE, University of Minnesota, Minneapolis,
                MN 55455, US}
}

\title{Consistent Energy-based Atomistic/Continuum Coupling for Two-body Potentials in Three Dimensions\thanks{Published in {\it SIAM J. Sci. Comput.}, {\bf 34}(3), B335--B360, 2012. The work was performed during the author's stay at the Chair of Computational Mathematics and Numerical Analysis (ANMC) at the Swiss Federal Institute of Technology (EPFL) whose support is acknowledged.}}

\newcommand{\cf}{cf.\thinspace}

\newcommand{\smfrac}[2]{{\textstyle\frac{#1}{#2}}}
\newcommand{\U}{{\mathcal U}}

\newcommand{\bbR}{{\mathbb R}}
\newcommand{\bbN}{{\mathbb N}}
\newcommand{\bbZ}{{\mathbb Z}}
\newcommand{\calL}{{\mathcal L}}
\newcommand{\calU}{{\mathcal U}}
\newcommand{\calR}{{\mathcal R}}
\newcommand{\calB}{{\mathcal B}}
\newcommand{\calT}{{\mathcal T}}

\newcommand{\calO}{{\mathcal O}}
\newcommand{\calX}{{\mathcal X}}

\newcommand{\C}{{\rm C}}
\renewcommand{\L}{{\rm L}}
\newcommand{\W}{{\rm W}}

\def\del{\delta\hspace{-1pt}}
\def\ddel{\delta^2\hspace{-1pt}}

\def\<{\langle}
\def\>{\rangle}

\newcommand{\delE}{\del E}

\newcommand{\Dc}[1]{\nabla_{\!#1}}
\newcommand{\Da}[1]{D_{\hspace{-1pt}#1}}

\def\mF{{\sf F}}
\def\mM{{\sf M}}

\def\hA{{\hat A}}
\def\hB{{\hat B}}
\def\hC{{\hat C}}
\def\hD{{\hat D}}

\newcommand{\cchi}{{\chi_{\mathstrut}}}

\newcommand{\eps}{\epsilon}

\def\ang{{\rm ang}}
\def\diam{{\rm diam}}
\def\c{{\rm c}}
\def\a{{\rm a}}

\def\D{{\rm D}}
\def\P{{\rm P}}
\def\dd{{\rm d}}

\def\db{{\rm db}}

\newcommand{\calTh}{{{\mathcal T}_h}}

\def\Xint#1{\mathchoice
{\XXint\displaystyle\textstyle{#1}}{\XXint\textstyle\scriptstyle{#1}}{\XXint\scriptstyle\scriptscriptstyle{#1}}{\XXint\scriptscriptstyle\scriptscriptstyle{#1}}\!\int}
\def\XXint#1#2#3{{\setbox0=\hbox{$#1{#2#3}{\int}$ }
\vcenter{\hbox{$#2#3$ }}\kern-.6\wd0}}
\def\mint{\Xint-}

\newcommand{\len}[1]{{\rm BondVol}(#1)}
\def\mod{\bmod}
\newcommand{\Trap}{{\rm Trap}}
\newcommand{\sgn}{{\rm sgn}}

\begin{document}
\sloppy
\maketitle

\begin{abstract}
Very few works exist to date on development of a consistent energy-based coupling of atomistic and continuum models of materials in more than one dimension.
The difficulty in constructing such a coupling consists in defining a coupled energy whose minimizers are free from uncontrollable errors on the atomistic/continuum interface.
In this paper a consistent coupling in three dimensions is proposed.
The main achievement of this work is to identify and efficiently treat a modified Cauchy--Born continuum model which can be coupled to the exact atomistic model.
The convergence and stability of the method is confirmed with numerical tests.
\end{abstract}



\pagestyle{myheadings}
\thispagestyle{plain}
\markboth{ALEXANDER V. SHAPEEV}{CONSISTENT ATOMISTIC/CONTINUUM COUPLING IN 3D}

\section{Introduction}

Modeling defects in crystalline solids requires using atomistic models.
On the one hand, defects create long-range elastic fields, the accurate resolution of which requires a huge number of atomistic degrees of freedom, often unhandleable even on modern computers.
On the other hand, the elastic deformation far away from a defect is well described by a continuum model.
This is a rationale for {\it atomistic/continuum} (A/C) coupled methods---the methods that use full atomistic resolution around a defect, coupled to a coarse-grained continuum model far from the defect \cite{MillerTadmor2002, MillerTadmor2009, VanLiLuskinEtAl}.

Consider the problem of finding an equilibrium of a certain atomistic crystal with a localized defect, i.e., finding a critical point of a potential energy of such a crystal.
An A/C coupling approach to this problem would be to consider the exact energies of the atomistic deformation near the defect and a Cauchy--Born continuum energy (cf.\ \cite{BornHuang1954}) of a $P_1$ finite element discretization of the deformation field far from the defect.
The efficiency of an A/C coupling rests on the fact that the complexity of computing the energy and the effective forces associated with an element is independent of the size of the element (which would not be true if the full atomistic model were used everywhere).

The two main variants of an A/C coupling are the energy-based and the force-based coupling, the first defines an A/C coupled energy that depends on the atomistic and continuum deformation, while the second mixes the atomistic and the continuum forces (i.e., derivatives of the energy of the atomistic model and the continuum model); see \cite{MillerTadmor2009, Shapeev2011, VanLiLuskinEtAl} and the references therein for more details.
The force-based coupling can indeed effectively approximate the exact atomistic equilibrium equations; however, its stability properties are not well understood at present \cite{DobsonLuskinOrtner2010a, DobsonLuskinOrtner2010b, MillerTadmor2009}, and indeed there seem to exist examples of a force-based coupling of stable atomistic and continuum equations being unstable \cite{Shapeev:qcf_stability:in_progress}.

When using energy-based methods, one faces a different kind of challenge: it turns out to be difficult to design a coupling which is at least first-order consistent (a first order consistency is equivalent to a first-order convergence provided that there is stability).\footnote{In the engineering-oriented literature, a lack of consistency is formulated in terms of fictitious forces called ``ghost forces.''}
Despite the recent progress in designing a consistent energy-based A/C coupling \cite{ELuYang2006, IyerGavini2011, OrtnerZhang, Shapeev2011, ShimokawaMortensenSchiotzEtAl2004}, no satisfactory solution exists in the general case to date.
We should also note that it has been recently shown that the so-called {\it blended} coupling methods (which assume a smooth transition between the atomistic and the continuum models) can address the stability issues of the force-based coupling \cite{LiLuskinOrtner, LuMing} and the consistency issues of the energy-based coupling \cite{LuskinVanKoten, LuskinOrtnerVanKoten}.

One of the recent developments
is the work \cite{Shapeev2011}, where the author proposed a consistent A/C coupling for two-body interaction in two dimensions (2D).
The key instrument in constructing a consistent coupling in \cite{Shapeev2011} was the two-dimensional {\it bond density lemma}, which asserts that the effective number of atomistic bonds in a certain direction $r\in\bbZ^2$ lying on any triangle with vertices restricted to the lattice $\bbZ^2$ is equal to the area of the triangle, regardless of the direction $r$.
This lemma allows one to define the A/C coupling method in terms of the energy of individual bonds and the show that continuum approximations of bond energies sum up to a (discretized) Cauchy--Born energy, up to some correction near the interface.

The purpose of this paper is to extend the method of \cite{Shapeev2011} to the three-dimensional case.
Unfortunately, the three-dimensional analogue of the bond density lemma is not true: the number of bonds lying in a tetrahedron depends on the bond direction and in general is not equal to the volume of the tetrahedron.
This makes the extension to three dimensions (3D) not trivial.

The construction of the method in 3D is similar to the lower-dimensional construction: we first define the continuum energy of bonds consistent with the exact energy of the bonds and then show that the sum of continuum energies of bonds can be computed efficiently.
The resulting three-dimensional continuum model turns out to be different from the Cauchy--Born model (this is a consequence of the lack of the three-dimensional bond density lemma).
Nonetheless, numerical tests conducted confirm a computational efficiency and accuracy similar to that in 2D \cite{OrtnerShapeev:qce_analysis:preprint, Shapeev2011}.

The paper is organized as follow.
We formulate the proposed A/C coupling in Section \ref{sec:3d_coupling} and define the effective number of bonds within a tetrahedron, $\len{T,r}$; efficient computation of this quantity is central to the overall efficiency of the proposed method.
Section \ref{sec:computing_Len} is dedicated entirely to an efficient algorithm for computing $\len{T,r}$, and a Matlab code of this algorithm is given in Appendix \ref{sec:matlab}.
In Section \ref{sec:num} we present numerical tests of accuracy and stability, and in Section \ref{sec:conclusion} we give concluding remarks.

\section{Consistent A/C Coupling}\label{sec:3d_coupling}

For generality, we present the A/C coupling in $\bbR^d$, but will focus mainly on case $d=3$.
The cases $d=2$ (considered in \cite{Shapeev2011}) and $d=1$ (considered in \cite{LiLuskin2011_finite.range, Shapeev2011}) will be particular cases of the discussion below. 

In Section \ref{sec:3d_coupling:geom} we define the continuum and discrete regions, the triangulation of the continuum region, and the corresponding functional spaces.
In Section \ref{sec:3d_coupling:atomistic_model} we present an atomistic interaction in terms of bond energies.
Finally, in Section \ref{sec:3d_coupling:coupling} we formulate the proposed A/C coupling.

\subsection{A/C Coupling Geometry} \label{sec:3d_coupling:geom}

\begin{figure}
\begin{center}
\includegraphics{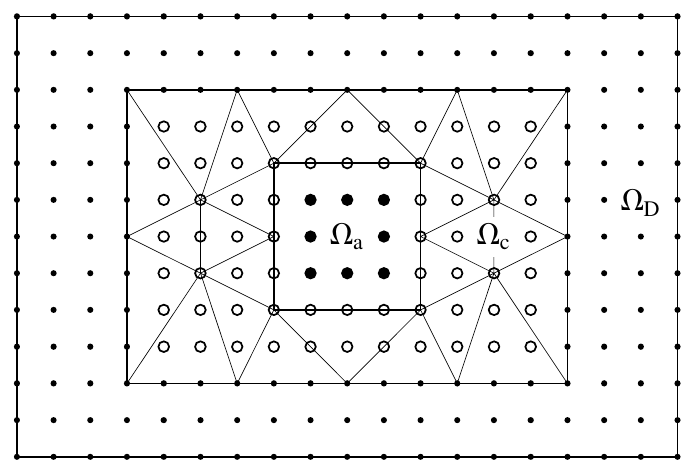}
\end{center}
\caption{Geometry of an A/C interface.
	Black atoms belong to the (discrete) atomistic region $\Omega_\a$, and white atoms belong to the continuum region $\Omega_\c$.
	The small atoms, belonging to $\Omega_\D$, are involved in imposing the Dirichlet-type boundary conditions.
	The discrete domains are, respectively, $\calL_\a$, $\calL_\c$, and $\calL_\D$.
}
\label{fig:ac-geom}
\end{figure}

Consider an atomistic material occupying a bounded domain $\Omega \subset \bbR^d$ in its undeformed (reference) state.
Assume a splitting of $\Omega$ into three (open) regions: $\Omega_\a$, where the exact atomistic model will be used; $\Omega_\c$, where a continuum approximation will be used; and $\Omega_\D$ containing atoms whose positions will be fixed (when posing Dirichlet-type boundary conditions); see Figure \ref{fig:ac-geom} for an illustration.
The atom positions in the undeformed state comprise the lattice $\calL = \overline{\Omega}\cap \bbZ^d$ (where $\overline{\mathstrut\bullet\mathstrut}$ denotes a closure of a set).
We also denote $\calL_\D = \calL\cap\overline\Omega_\D$, $\calL_\c = (\calL\cap\overline\Omega_\c)\setminus\calL_\D$, and $\calL_\a = \calL\setminus(\calL_\D\cup\calL_\c)$.
Normally, the number of atoms in $\calL_\a$ is much less than the total number of atoms.

It should be stressed that although we assume perfect crystalline lattices without defects in the reference configuration, the computational method can be presented if defects are allowed in the atomistic region.

We define the space of admissible deformations, $\calU$, as a set of discrete functions $\calL\to\bbR^d$ whose values on $\calL_\D$ are set according to the uniform deformation gradient $\mF\in\bbR^{d\times d}$, i.e.
\[
\calU := \{u:\calL\to\bbR^d :\, u(x) = \mF x~\forall x\in\calL_\D \}
,
\]
and we define the space of admissible displacements as
\[
\calU_0 := \{u\in\calU :\, u|_{\calL_\D}=0\}
.
\]
Additional assumptions on $\calL_\D$ will be made later to avoid unnecessary complications due to boundary effects.

We assume that $\Omega_\c$ is a polytope (i.e., polyhedron for $d=3$) and $\calTh$ is its triangulation with simplices $T\in\calTh$.
The spaces of A/C deformations and admissible displacements are defined as
\begin{align*}
\calU^h = \{u^h : \overline\Omega_\c\cup\calL_\a \to \bbR^d : 
~& u(x) = \mF x~\forall x\in\calL_\D,
\\ &\text{$u^h$ is continuous on $\Omega_\c$ and is affine on each $T\in\calTh$} \},
\end{align*}
and $\calU^h_0 := \{u^h\in\calU^h :\, u^h|_{\calL_\D}=0\}$.

At this stage, we do not require that the vertices of simplices $T\in\calTh$ lie on the lattice $\calL$, as was assumed in \cite{Shapeev2011}.
The method will be formulated for a general triangulation $\calTh$; however, the algorithm (Section \ref{sec:computing_Len}) will be developed under such an assumption.

\subsection{Bond Formulation of the Atomistic Model}\label{sec:3d_coupling:atomistic_model}

We assume that the atomistic interaction is given by a set of neighbors $\calR\subset\bbZ^d\setminus\{0\}$ and a two-body potential $\phi:\bbR^d\to\bbR$.
Define an interval $(x, x+r)$ between two points $x, x+r\in\bbR^d$ as a set
\[
(x, x+r) := \left\{x+\lambda r:\, \lambda\in(0,1)\right\}
,
\]
and call it a {\it bond;} $r$ will be called a {\it direction} of a bond.
Introduce the finite difference associated with a bond $b=(x,x+r)$ or a bond direction $r$, i.e.,
\[
\Da{b} y := \Da{r}\,y(x) := y(x+r)-y(x),
\]
and let the energy of the atomistic model be
\begin{equation} \label{eq:manyD:E-definition} 
E(y) := \sum_{b\in\calB} \phi(\Da{b} y),
\end{equation}
where
\[
\calB = 
	\{
		(x,x+r) :\, 
		x,x+r\in\calL,~r\in\calR
	\}
\]
is the collection of all bonds in the system, and $\phi(z)$ is the energy of interaction of two atoms with $z\in\bbR^d$ being the position of one atom relative to another atom.
The variational equilibrium condition for $y\in\calU$ (under no external force) is then
\[
\<\delE(y), v\>=0
\quad\forall v\in\calU_0
,
\]
where $\<\bullet,\bullet\>$ is the scalar product in $\calU$, and $\delE(y)\in\calU$ is the G\^ateaux derivative of $E$ defined as $\<\delE(y), v\> := \frac\partial{\partial\alpha} E(y+\alpha v)\big|_{\alpha=0}$.
Mechanically, $\delE(y)$ are the interaction forces on the deformation $y$.

We assume that ${\rm dist}(\partial\Omega, \Omega\setminus\Omega_\D) \geq \max_{r\in\calR} |r|$ so that the following discrete version of the divergence theorem holds:
\[
\sum_{x\in\calL\cap(\calL-r)} D_r v(x) = 0
\quad \forall v\in\U_0,~ \forall r\in\calR.
\]
It is then easy to verify (see \cite[eq.\ (2.6)]{Shapeev2011}) that a uniform deformation $y_\mF(x) := \mF x$ is an equilibrium; i.e., $\<\delE(y_\mF), v\>=0$.

In the next subsection we propose an A/C coupling $E^h(y)$, which is an approximation of $E(y)$, such that
\begin{equation}\label{eq:Eh_is_consistent}
\del E^h(y_\mF)=0;
\end{equation}
in other words, $E^h(y)$ is {\it patch-test consistent}.
This method will be a generalization of the one-dimensional method \cite{LiLuskin2011_finite.range, Shapeev2011} and the two-dimensional method \cite{Shapeev2011} (more precisely, the version of the method labeled as ECC in \cite{Shapeev2011}).

\subsection{The Proposed A/C Coupling}\label{sec:3d_coupling:coupling}

Define the set of continuum bonds $\calB_\c := \{b\in\calB \,:~ b\subset\Omega_\c\}$ and the continuum directional derivative (associated with the bond direction $r$)
\[
\Dc{r}\,y(x) := \lim_{\eps\to0} \big(f(x+\eps r)-f(x)\big).
\]
Further, define averaging over a bond $b=(x,x+r)$:
\[
\mint_{x}^{x+r} f(x) \,\db = \mint_b f(x) \,\db := \int_{0}^1 f(x+\lambda r) \dd\lambda.
\]

The proposed (patch-test) consistent A/C coupling then reads as
\begin{equation} \label{eq:Eh-def}
E^h(y) := \sum_{b\in\calB_\a} \phi(\Da{b} y) + \sum_{b\in\calB_\c} \mint_b \phi(\Dc{r_b}y) \db
,
\end{equation}
where $\calB_\a := \calB\setminus\calB_\c$ and $r_b$ denotes the direction of $b\in\calB$.
The formulation of an A/C coupling in terms of bond averages was first proposed in \cite{Ortner2011} for the quasinonlocal quasicontinuum method \cite{ShimokawaMortensenSchiotzEtAl2004} in one dimension.
The patch-test consistency (condition \eqref{eq:Eh_is_consistent}) for this coupling follows from the fact that the variation of the continuum energy of a bond $b\in\calB_\c$ is equal to the variation of the exact energy on a uniform deformation $y_\mF$:
\[
\big\<\del\big(\textstyle \mint_b \phi(\Dc{r_b}y_\mF)\big)\db, v\big\>
=
\mint_b \phi'(\mF r_b)\cdot \Dc{r_b} v \db
=
\phi'(\mF r_b)\cdot \Da{r_b} v
=
\big\<\del \phi(\Da{r_b}y_\mF), v\big\>,
\]
where $\phi'$ denotes the gradient of $\phi$.
The proof of patch-test consistency is completely analogous to the proof of the two-dimensional version of this statement \cite[Prop.\ 3.2]{Shapeev2011}.

A straightforward calculation (see Proposition \ref{prop:bond_to_volume} below) allows us to convert bond integrals in \eqref{eq:Eh-def} into a sum over elements (i.e., effectively, volume integrals):
\begin{equation}\label{eq:coupled_energy_volumetric}
E^h(y) = \sum_{b\in\calB_\a} \phi(\Da{b} y) + \sum_{T\in\calTh} \sum_{r\in\calR} \Omega_{T,r} \phi(\Dc{r}y|_T),
\end{equation}
where the effective volumes of $T$ are defined as
\begin{equation}
\label{eq:OmegaTr}
\Omega_{T,r} := \sum_{\substack{x\in\bbZ^d \\ (x,x+r)\in\calB_\c}} \mint_x^{x+r} \cchi_T \db
\end{equation}
and the characteristic function $\chi_\bullet$ is defined in Section \ref{sec:characteristic_function}.

It should be noted that the second term in \eqref{eq:coupled_energy_volumetric} differs from the standard Cauchy--Born energy (for it to be equal to the Cauchy--Born energy, we must have $\Omega_{T,r}\equiv |T|$).
This in particular means that the existing results on stability \cite{EMing2007static, HudsonOrtner, OrtnerShapeev:qce_analysis:preprint} and first-order consistency \cite{Ortnera, OrtnerShapeev:qce_analysis:preprint, OrtnerZhang} of the Cauchy--Born model do not directly apply to the coupling proposed in the present paper.
However, the consistency analysis of \cite{OrtnerShapeev:qce_analysis:preprint} should be applicable to the proposed coupling as this analysis is based mainly on the bond formulation \eqref{eq:Eh-def} and does not require the exact bond density lemma (i.e, we would need only a result of the form $\Omega_{T,r} \leq c |T|$).
In the present paper, convergence and stability of the proposed coupling are confirmed via numerical tests in Section \ref{sec:num}.

\subsubsection{Characteristic Function}\label{sec:characteristic_function}

For a polytope $\omega \subset \bbR^d$ (e.g., polyhedron in 3D) define a characteristic function
\begin{equation}\label{eq:chi_def}
\chi_\omega(x) := \lim_{\rho\to0} \frac{|\omega\cap B_\rho(x)|}{|B_\rho(x)|},
\end{equation}
where $B_\rho(x)$ is the ball centered at $x$ with the radius $\rho$.
We note that (i) the limit w.r.t.\ $\rho\to 0$ in the definition of $\cchi_{\omega}(x)$ exists, and (ii) including/excluding the boundary of a polytope $\omega$ (or any part of it) does not change the point values of $\cchi_{\omega}(x)$.
The characteristic function is defined so that if $\overline{\omega} = \overline{\omega_1} \cup \overline{\omega_2}$ and $\omega_1\cap\omega_2=\emptyset$, then $\chi_\omega = \cchi_{\omega_1} + \cchi_{\omega_2}$ pointwise.
In particular, we have
\begin{equation} \label{eq:Omegac-partition}
\cchi_{\Omega_\c}(x) = \sum\limits_{T\in\calTh} \cchi_T(x)
\quad\forall x\in\bbR^d
.
\end{equation}

The characteristic function of a 3D polyhedron $\omega$ can be visualized as
\begin{equation}
\cchi_{\omega}(x) =
\begin{cases}
1           &\textnormal{if $x\in$ interior of $\omega$},
\\
\frac{1}{2} &\textnormal{if $x\in$ face of $\omega$},
\\
\frac{\alpha}{2\pi} &\textnormal{if $x\in$ edge of $\omega$ with angle $\alpha$},
\\[0.2em]
\frac{\beta}{4\pi} &\textnormal{if $x$ is a vertex of $\omega$ with spherical angle $\beta$},
\\
0 & \textnormal{otherwise}.
\end{cases}
\label{eq:chi-3d}
\end{equation}
Note that the values of $\cchi_{\omega}(x)$ at the vertices of $\omega$ will not be important for the formulation of the method.

With this definition of the characteristic function, we can prove the following proposition.
\begin{proposition}\label{prop:bond_to_volume}
The energy \eqref{eq:Eh-def} is equivalently written as \eqref{eq:coupled_energy_volumetric}.
\end{proposition}
\begin{proof}
Indeed, the second sum in \eqref{eq:Eh-def} can be transformed as
\begin{align*}
\sum_{b\in\calB_\c} \mint_b \phi(\Dc{r_b}y) \db
=~&
\sum_{r\in\calR} \sum_{\substack{x\in\bbZ^d \\ (x,x+r)\subset\Omega_\c}} \mint_x^{x+r} \phi(\Dc{r}y) \db
\\=~&
\sum_{r\in\calR} \sum_{\substack{x\in\bbZ^d \\ (x,x+r)\subset\Omega_\c}}
	\mint_x^{x+r} \cchi_{\Omega_\c} \phi(\Dc{r}y) \db
\\=~&
\sum_{r\in\calR} \sum_{\substack{x\in\bbZ^d \\ (x,x+r)\subset\Omega_\c}}
	\mint_x^{x+r} \sum_{T\in\calT_h} \cchi_{T} \phi(\Dc{r}y) \db
\\=~&
\sum_{T\in\calT_h} \sum_{r\in\calR} \phi(\Dc{r}y|_T) \sum_{\substack{x\in\bbZ^d \\ (x,x+r)\subset\Omega_\c}}
	\mint_b \cchi_{T} \db
,
\end{align*}
where we used \eqref{eq:Omegac-partition} and the fact that $\cchi_{\Omega_\c}=1$ on any bond which lies inside $\Omega_\c$.
\end{proof}

\subsubsection{Complexity of Computing $E^h$}

The method \eqref{eq:Eh-def} with precomputing $\Omega_{T,r}$ directly using \eqref{eq:OmegaTr} may already yield a significant reduction in the number of operations.
Indeed, one must spend $\calO(\#\calB_\c)$ operations (here $\#$ denotes the number of elements in a set) on precomputing $\Omega_{T,r}$ only once for a given A/C geometry, and it would then take $\calO(\#\calB_\a)+\calO((\#\calT_h)(\#\calR))$ operations for computing the forces or assembling the stiffness matrix corresponding to \eqref{eq:Eh-def} (recall that $\#\calB_\a \ll \#\calB_\c$).

Furthermore, in the one-dimensional case and in the two-dimensional case with triangulation nodes coinciding with the lattice sites, the bond density lemma yields that $\Omega_{T,r} = |T|$ if $T$ is far enough from the A/C interface, and thus $E_\c(y)$ reduces to the standard Cauchy--Born energy up to an interface correction.
This removes the need for the precomputation step and yields an algorithm with an optimal complexity $\calO(\#\calB_\a)+\calO((\#\calT_h)(\#\calR))$.

Unfortunately, as also shown in \cite{Shapeev2011}, in general $\Omega_{T,r} \ne |T|$ in 3D.
(Numerical calculations of $\Omega_{T,r}$ with randomly generated $T$ and $r$ show that $\Omega_{T,r} \ne |T|$ for most choices of $T$ and $r$.)
Nevertheless, as will be shown in the present paper, one can design an algorithm for computing $\Omega_{T,r}$ efficiently.

To this end denote, for any polytope $\omega$,
\begin{equation}
\label{eq:len}
\len{\omega,r}
:=
\sum_{x\in\bbZ^3} \mint_x^{x+r} \chi_\omega \db
\end{equation}
so that $\Omega_{T,r}$ can be expressed through $\len{T,r}$ and the interface correction:
\begin{equation}\label{eq:OmegaTr_computation}
\Omega_{T,r} = \len{T,r} - \sum_{(x,x+r)\in\calB_\a} \mint_x^{x+r} \cchi_T \db
.
\end{equation}
Here the quantity $\len{\omega,r}$ is an effective number (volume) of all the bonds with direction $r$ that intersect a polytope $\omega$.
In section \ref{sec:computing_Len} we will show that $\len{T,r}$ can be computed with $\calO(\log(\diam(T)) + \log|r|)$ arithmetic operations in 3D, assuming that the vertices of $T$ lie on the lattice $\bbZ^3$.
This implies the following result.
\begin{theorem}\label{prop:algorithm}
Consider the following algorithm.
\begin{itemize}
\item[1(a).] Precompute $\len{T,r}$ as described in Sections \ref{sec:computing_Len:reductions}--\ref{sec:computing_Len:Sab} for all $T\in\calT_h$ and $r\in\calR$.
\item[1(b).] Precompute $\Omega_{T,r}$ by the formula \eqref{eq:OmegaTr_computation} for all $T\in\calT_h$ and $r\in\calR$.
\item[2.] Compute $E^h(y)$ by the formula \eqref{eq:coupled_energy_volumetric}.
\end{itemize}
Then the complexity of step 1(a) is
\begin{equation}\label{eq:algorithm:complexity}
\calO\big((\#\calTh)(\#\calR)\,(\log(\diam(\Omega))+\log(\diam(\calR)))\big)
,
\end{equation}
and the complexity of step 2 is $\calO(\#\calB_\a + (\#\calT_h)(\#\calR))$.
\end{theorem}
\begin{proof}
The proof of complexity estimate for step 1a follows directly from Proposition \ref{prop:complexity_of_len}, and the complexity estimate for step 2 is evident.
\end{proof}
\begin{remark}
A straightforward algorithm for computing $\Omega_{T,r}$ has the complexity of $\calO((\#\calB_\a)(\#\calTh))$.
However, it is possible to formulate an algorithm with complexity $\calO((\#\calB_\a)\diam(\calR))$, provided that the number of triangles crossing a bond $b$ can be estimated by $C|b|$, where $C$ depends only on the shape regularity of $\calTh$ (refer to \cite[Lemma 5.7]{OrtnerShapeev:qce_analysis:preprint} for a related two-dimensional result).
Additional optimization can be done by taking into consideration the fact that only bonds near the interface can change $\Omega_{T,r}$.
\end{remark}

We expect that the precomputation time will not overly dominate the main computation time in most of applications.
Indeed, the factor $\calO\big(\log(\diam(\Omega))+\log(\diam(\calR)))$ is essentially the maximal number of iterations of Euclid's algorithm and is between $15$ and $50$ for a typical atomistic system with $\diam(\calR)\approx 5$ and $10^2\lessapprox\diam(\Omega)\lessapprox10^9$.
In the numerical tests conducted in this paper, computing the bond volumes (step 1(a)) was several times faster than doing the interface correction (step 1(b)).

\section{Computing $\len{T,r}$}\label{sec:computing_Len}

This section is devoted entirely to an algorithm of fast computation of $\len{T,r}$ defined by \eqref{eq:len} assuming that the vertices of a polyhedron $T$ belong to a lattice $\bbZ^3$ (however, many ingredients of the algorithm, in particular the entire Section \ref{sec:computing_Len:reductions}, remain true for any $T$).
A reader who is not interested in details or justification of the algorithm can skip this section or refer to Appendix \ref{sec:matlab} for a Matlab code.

\begin{table}
  \begin{center}
  \begin{tabular}{|c|c|} \hline
    Reduction to: & \#(parameters) \\ \hline
    $\len{T,r}$ & 15 \\
    $\gcd(r_1,r_2,r_3)=1$ & 15 \\
    $r=e_3$ & 12 \\
    truncated prism & 9 \\
    triangle & 6 \\
    trapezoid & 4 \\
    right triangle & 2 \\
    $S_{a,b}$ & 2 \\ \hline
  \end{tabular}
  \caption{\label{tbl:algorithm} 
  	List of reductions from the original problem of computing $\len{T,r}$ to computing $S_{a,b}$, with the number of parameters left after the reduction.
  }
  \end{center}
\end{table}

The algorithm is based on a series of steps, presented in Sections \ref{sec:computing_Len:reductions}--\ref{sec:computing_Len:Sab}, that reduce the original problem of computing $\len{T,r}$ with 15 scalar parameters (12 to define $T$ and 3 to define $r$) to an integer sum $S_{a,b}$ (cf.\ \eqref{eq:Sab}) with only 2 parameters $a,b\in\bbZ$.
The principal steps of the algorithm are summarized in Table \ref{tbl:algorithm}.
The overall complexity of the algorithm is discussed in Section \ref{sec:computing_Len:complexity}.

In this section, by $A,B,C,D,X,Y,Z$ we will denote the points in $\bbR^3$, and by $r, s, x, \xi$ we will denote vectors in $\bbR^3$.
The points may be identified with their radius vectors.
For two points $X,Y\in\bbR^3$, we denote $\overrightarrow{XY} = Y-X$.
For points and vectors, the subscripts $1$, $2$, and $3$ will denote their coordinates, and we will use the notation $X=(X_1, X_2, X_3)$.
The standard basis of $\bbR^3$ will be denoted by $e_1,e_2,e_3$.
The points on the $xy$-plane in $\bbR^3$ (i.e., the plane $\{X\in\bbR^3:\,X_3=0\}$) will be identified with the respective points in $\bbR^2$.

\subsection{Reductions}\label{sec:computing_Len:reductions}

We start with a tetrahedron $T$ and a vector $r=(r_1,r_2,r_3)\in\bbZ^3$, $r\ne 0$.

\subsubsection{Reduction to the case $\gcd(r_1, r_2, r_3)=1$}\label{sec:computing_Len:reductions:gcd}

In this paragraph we show that
\begin{equation}
\label{eq:reduction1:claim}
\len{T,r} = \len{T,r/\gcd(r_1, r_2, r_3)},
\end{equation}
where $\gcd(r_1, r_2, r_3)$ is the greatest common divisor of $r_1, r_2, r_3\in\bbZ$.

Indeed, let $r=n s$ with $n=\gcd(r_1,r_2,r_3)\in\bbN$ and $s\in\bbZ^3$.
Fix $x\in\bbZ^3$, and consider a collection of points
\begin{equation}
\label{eq:calX}
\calX_x = \{x+is \,:~ i\in\bbZ\}
.
\end{equation}
First, compute the contribution of a collection of the respective collection of bonds $\{(\xi,\xi+s) \,:~ \xi\in \calX_x\}$ to $\len{T,s}$:
\begin{equation}
\label{eq:reduction1:calc1}
\sum_{\xi\in\calX_x} \mint_{\xi}^{\xi+s} \cchi_T \db
=
\sum_{i\in\bbZ} \mint_{x+is}^{x+is+s} \cchi_T \db
=
\sum_{i\in\bbZ} \int_0^1 \cchi_T(x+(i+\lambda)s) \dd\lambda
=
\int_{-\infty}^{\infty} \cchi_T(x+\lambda s) \dd\lambda.
\end{equation}
Second, compute the contributions of
\begin{align*}
\{(\xi,\xi+r) \,:~ \xi\in \calX_x\}
=~&
\{(x+is,x+is+r) \,:~ i\in\bbZ\}
\\ =~&
\{(x+js+ir,x+js+ir+r) \,:~ i\in\bbZ,~j=0,1,\ldots,n-1\}
\end{align*}
to $\len{T,r}$:
\begin{equation}
\label{eq:reduction1:calc2}
\begin{split}
\sum_{\xi\in\calX_x} \mint_{\xi}^{\xi+r} \cchi_T \db
=~&
\sum_{j=0}^{n-1}
\sum_{i\in\bbZ} \mint_{x+js+ir}^{x+js+ir+r} \cchi_T \db
\\ =~&
\sum_{j=0}^{n-1}
\int_{-\infty}^{\infty} \cchi_T(x+js+\lambda r) \dd\lambda
\\ =~&
\frac1n \sum_{j=0}^{n-1}
\int_{-\infty}^{\infty} \cchi_T(x+\mu s) \dd\mu
\\ =~&
\int_{-\infty}^{\infty} \cchi_T(x+\mu s) \dd\mu,
\end{split}
\end{equation}
where we did the change of variable $\mu=j+\lambda n$.

From calculations \eqref{eq:reduction1:calc1} and \eqref{eq:reduction1:calc2} it is easy to see that $\len{T,s}=\len{T,r}$.
Indeed, summing the contributions of different $\calX_x$ yields
\begin{align}
\label{eq:len-through-calX}
\len{T,s}
=~&
\sum_{\calX_x} \sum_{\xi\in\calX_x} \mint_{\xi}^{\xi+s} \cchi_T \db
\\ =~& \notag
\sum_{\calX_x} \sum_{\xi\in\calX_x} \mint_{\xi}^{\xi+r} \cchi_T \db
=
\len{T,r},
\end{align}
which proves \eqref{eq:reduction1:claim}.

\subsubsection{Reduction to the case $r=e_3$}\label{sec:computing_Len:reductions:r_e3}

We now assume $\gcd(r_1, r_2, r_3)=1$.
In this subsection we first find a suitable linear transformation $\mM$ such that $\mM r=e_3$, and apply it to both $T$ and $r$.
We then extend the definitions of $\len{\omega, r}$ and $\chi_\omega$ to allow measuring angles of edges and vertices in the untransformed space.

\medskip{\bf Construction of $\mM$}. Due to B\'ezout's lemma, there exist $c_{12}, d_{12},c_3, d_3\in\bbZ$ such that
\[
r_1 c_{12} + r_2 d_{12} = \gcd(r_1, r_2),
\qquad
\gcd(r_1, r_2) c_3 + r_3 d_3 = \gcd(\gcd(r_1, r_2), r_3) = 1.
\]
Take the matrix $\mM\in\bbZ^{3\times3}$ as the product of two matrices,
\[
\mM =
\begin{pmatrix}
r_3 & 0 & -\gcd(r_1, r_2) \\
0 & 1 & 0 \\
c_3 & 0 & d_3
\end{pmatrix}
\,
\begin{pmatrix}
c_{12} & d_{12} & 0 \\
-\frac{r_2}{\gcd(r_1, r_2)} & \frac{r_1}{\gcd(r_1, r_2)} & 0 \\
0 & 0 & 1
\end{pmatrix}
,
\]
and compute $\mM r$:
\[
\begin{pmatrix}
c_{12} & d_{12} & 0 \\
-\frac{r_2}{\gcd(r_1, r_2)} & \frac{r_1}{\gcd(r_1, r_2)} & 0 \\
0 & 0 & 1
\end{pmatrix}
\,
\begin{pmatrix}
r_1 \\ r_2 \\ r_3
\end{pmatrix}
=
\begin{pmatrix}
\gcd(r_1, r_2) \\ 0 \\ r_3
\end{pmatrix},
\]
\[
\mM r
=
\begin{pmatrix}
r_3 & 0 & -\gcd(r_1, r_2) \\
0 & 1 & 0 \\
c_3 & 0 & d_3
\end{pmatrix}
\,
\begin{pmatrix}
\gcd(r_1, r_2) \\ 0 \\ r_3
\end{pmatrix}
=
\begin{pmatrix}
0 \\ 0 \\ 1
\end{pmatrix}
=e_3
.
\]

It is also important to notice that both $\mM$ and $\mM^{-1}$ have integer coefficients; the latter is due to $\det\mM=1$.
Hence
\begin{equation} \label{eq:MZ3}
\mM \bbZ^3=\bbZ^3
.
\end{equation}

{\bf Extension of $\chi_\omega$ and $\len{\omega, r}$}. 
We need to apply the transformation $\mM$ to both $T$ and $r$.
To that end, we extend the definition of $\chi_\omega$ by allowing for measuring angles of edges and vertices of $\omega$ after applying $\mM$:
\[
\chi^\mM_{\omega}(x) := \cchi_{\mM^{-1}\omega}(\mM^{-1} x),
\]
so that $\chi_{\omega}(x) = \chi^\mM_{\mM \omega}(\mM x)$.
In the case when $\omega$ is a polyhedron, $\chi^\mM_{\omega}$ can be evaluated as
\[
\chi^\mM_{\omega}(x) =
\begin{cases}
1           &\textnormal{if $x\in$ interior of $\omega$},
\\
\frac{1}{2} &\textnormal{if $x\in$ face of $\omega$},
\\
\frac{\alpha}{2\pi} &\textnormal{if $x\in e$, $e$ is an edge of $\omega$, $\alpha$ is the angle of $\mM^{-1} e$ in $\mM^{-1}\omega$},
\\[0.2em]
\frac{\beta}{4\pi} &\textnormal{if $x$ is a vertex of $\omega$, $\beta$ is the spherical angle of $\mM^{-1} x$ in $\mM^{-1}\omega$},
\\
0 & \textnormal{otherwise}.
\end{cases}
\]
Likewise, extend
\begin{equation}
\label{eq:lenM}
\len{\mM, \omega, r} := \len{\mM^{-1} \omega, \mM^{-1} r}
= \sum_{x\in\mM\bbZ^3} \mint_{x}^{x+r} \chi^\mM_\omega \db
= \sum_{x\in\bbZ^3} \mint_{x}^{x+r} \chi^\mM_\omega \db
\end{equation}
so that $\len{\omega,r} = \len{\mM, \mM\omega, \mM r}$.
Note that in the last step of \eqref{eq:lenM} we used \eqref{eq:MZ3}.
It is worthwhile noting that $\len{\mM, \omega, r}$ is equal to $\len{\omega, r}$ unless $r$ is parallel to some edge of $\omega$.

Thus, we reduced
\[
\len{T,r} = \len{\mM, \mM T,\mM r} = \len{\mM, \mM T, e_3}.
\]

\subsubsection{Reduction to Truncated Prism}\label{sec:computing_Len:reductions:prism}

Denote the vertices of $T$ by $A,B,C,D\in\bbZ^3$, and choose their ordering to have a positive orientation in 3D, i.e., so that the vectors $\overrightarrow{AB}$, $\overrightarrow{AC}$, and $\overrightarrow{AD}$ form a positively orientated basis.

The tetrahedron $T$ can be represented as an oriented sum of truncated prisms, which can be rigorously expressed with characteristic functions:
\begin{equation}\label{eq:tet_split}
\cchi_{T} =
- o(B'C'D') \cchi_{\P(BCD)}
+ o(A'C'D')\cchi_{\P(ACD)}
- o(A'B'D')\cchi_{\P(ABD)}
+ o(A'B'C')\cchi_{\P(ABC)},
\end{equation}
where by $\bullet'$ we denote a projection of a point or a vector on the $xy$-plane (i.e., on the plane orthogonal to $e_3$), $\P(XYZ)$ is a truncated prism with vertices $X$, $Y$, and $Z$ and their projections $X'$, $Y'$, $Z'$ (see Figure \ref{fig:prism} for an illustration), and $o(XYZ)\in\{-1,0,1\}$ is an orientation of three points $X,Y,Z\in\bbZ^2$ defined to be zero if $X,Y,Z$ lie on the same line and to be equal to the orientation of the basis $\overrightarrow{XY}$, $\overrightarrow{XZ}$ otherwise.

\begin{figure}
\begin{center}
\includegraphics{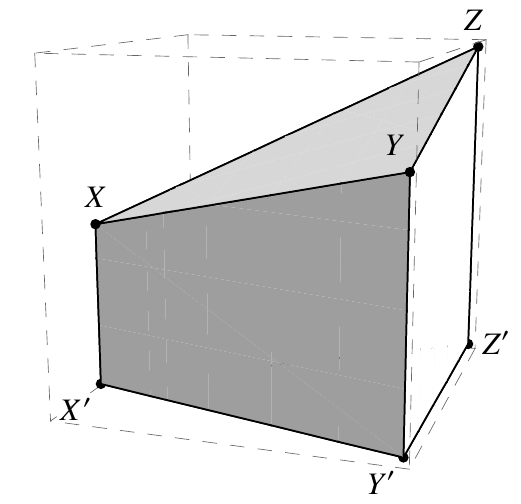}
\end{center}
\caption{A truncated prism $\P(XYZ)$ formed by three vertices $X$, $Y$, $Z$ and their projections $X'$, $Y'$ and $Z'$ on the $xy$-plane.
}
\label{fig:prism}
\end{figure}
The lower-dimensional version of \eqref{eq:tet_split} (i.e., splitting of a triangle into trapezia) is illustrated in Figure \ref{fig:triangle_splitting}, and the proof for an arbitrary dimension is given in Appendix \ref{sec:triangle_splitting_proof}.

For convenience, we assume that all four points $A$, $B$, $C$, and $D$ lie above the $xy$-plane (otherwise some prisms may be ill-defined).
Obviously, one can always shift $T$ upwards to satisfy this requirement.
This, however, is not required with an appropriate generalization of $\cchi_{\P(XYZ)}$ when $T$ is not entirely above the $xy$-plane; refer to Appendix \ref{sec:triangle_splitting_proof} for more details.

Thus, we reduced computing $\len{T, r}$ for a tetrahedron $T$ to computing
\begin{align*}
\len{\mM, T, r} =
   -~& o(B'C'D') \len{\mM, \P(BCD), r}
\\ +~& o(A'C'D') \len{\mM, \P(ACD), r}
\\ -~& o(A'B'D') \len{\mM, \P(ABD), r}
\\ +~& o(A'B'C') \len{\mM, \P(ABC), r}.
\end{align*}

\subsubsection{Reduction to Sums over Triangles}\label{sec:computing_Len:reductions:triangle}

We have that $\mM\in\bbZ^{3\times3}$, $\det\mM=1$, $r=e_3$, $P=\P(ABC)$ is a truncated prism, the points $A,B,C\in\bbZ^3$ are above the $xy$-plane, and we need to compute $\len{\mM,P,r}$.
In what follows we will use the notation $\triangle(XYZ)$ for a triangle with three vertices $X,Y,Z\in\bbR^2$.

We can assume that the plane $ABC$ is not parallel to the $z$ axis: otherwise $\P(ABC)$ is degenerate and $\len{\mM,P,r}=0$ (since then $\chi^\mM_P\equiv 0$).
Hence, let $z=c_1 x + c_2 y + c_3$ be an equation of the plane $ABC$.

We will split all the bonds, as we did in \eqref{eq:len-through-calX}, into the classes $(\xi,\xi+r)$, $\xi\in\calX_x$ (\cf \eqref{eq:calX}), with each class defined by $x=i e_1+j e_2$, $i,j\in\bbZ$.
That is, we express
\[
\len{P,r} =
\sum_{i,j\in\bbZ}
\sum_{\xi\in\calX_{(i,j,0)}}
\mint_{\xi}^{\xi+r} \chi^\mM_P \db
=
\sum_{i,j\in\bbZ}
\int_{-\infty}^\infty \chi^\mM_P(i,j,z) \dd z
=
\sum_{i,j\in\bbZ}
\int_0^{c_1 i+c_2 j+c_3} \chi^\mM_P(i,j,z) \dd z
.
\]
For $x_3$ between $0$ and $c_1 i+c_2 j+c_3$, we have $\chi^\mM_P(i,j,z)=1$ if $(i,j)$ is in the interior of $\triangle=\triangle(A'B'C')$, $\chi^\mM_P(i,j,z)=\smfrac12$ if $(i,j)$ is on the edge of $\triangle$, and $\chi^\mM_P(i,j,z)=\smfrac\alpha{2\pi}$ if $(i,j)$ is on the vertex of $\triangle$.
The value $\alpha$ is determined by the edges $v$ and $w$ sharing the respective vertex: $\alpha$ is equal to the angle between the plane formed by $\mM^{-1}v$ and $\mM^{-1}e_3$ and the plane formed by $\mM^{-1}w$ and $\mM^{-1}e_3$.
The latter is equal to the angle between $\mM^{-1}v\times \mM^{-1}e_3$ and $\mM^{-1}w\times \mM^{-1}e_3$.

Thus, if we define, for a polygon $S\subset\bbR^2$, its characteristic function
\[
\tilde\chi^{\mM}_S(i,j) =
\begin{cases}
1                   & \text{ if $(i,j)\in$ interior of S,} \\
\smfrac12           & \text{ if $(i,j)\in$ edge of S,} \\
\smfrac\alpha{2\pi} & \text{ if $(i,j)\in$ vertex of S sharing edges $v,w$, where}
                 \\ & \text{ ~~$\alpha$ is the angle between $\mM^{-1}v\times \mM^{-1}e_3$ and $\mM^{-1}w\times \mM^{-1}e_3$},
\end{cases}
\]
then 
\[
\int_0^{c_1 i+c_2 j+c_3} \chi^\mM_P(i,j,z) \dd z
=
(c_1 i+c_2 j+c_3) \tilde\chi^{\mM}_{\triangle(A'B'C')}(i,j),
\]
and hence
\begin{align}\notag
o(A'B'C')\len{P,r} =~&
o(A'B'C')\sum_{i,j\in\bbZ}
(c_1 i+c_2 j+c_3) \tilde\chi^{\mM}_{\triangle(A'B'C')}(i,j)
\\ =:~& \label{eq:Sprime}
S'_{A'B'C'}(\mM; c_1,c_2,c_3)
.
\end{align}
We thus reduced the problem to computing the sum $S'_{A'B'C'}(\mM; c_1,c_2,c_3)$ over a triangle $\triangle(A'B'C')$.

\subsubsection{Reduction to a Sum over Right Triangles}\label{sec:computing_Len:reductions:right_triangle}

Let $\triangle=\triangle(ABC)$, and let $A'$, $B'$, and $C'$ be the projections of $A, B, C \in\bbZ^2$ on the $x$-axis (i.e., on $e_1$), as illustrated in the Figure \ref{fig:triangle_splitting}.

\begin{figure}
\begin{center}
\hfill
\subfigure[]{\label{fig:triangle_splitting}\includegraphics{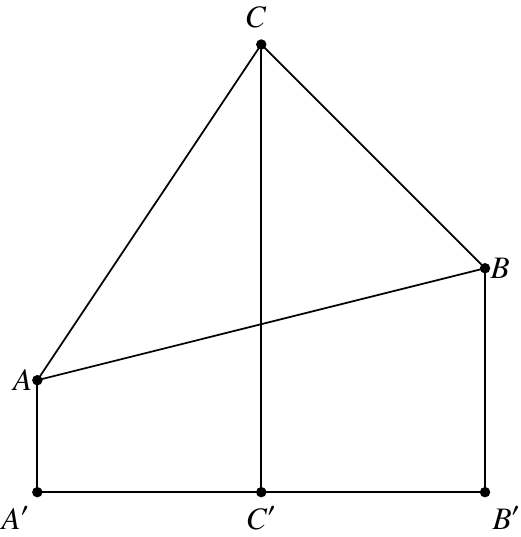}}
\hfill\hfill
\subfigure[]{\label{fig:trapezoid_splitting}\includegraphics{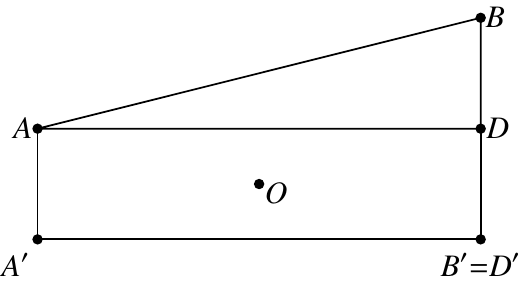}}
\hfill$\mathstrut$
\end{center}
\caption{Left: a triangle $\triangle(ABC)$ and the projections of its vertices on the $x$-axis. The triangle can thus be represented as an oriented sum of three trapezia, $\Trap(AB)$, $\Trap(BC)$, $\Trap(CA)$. One can notice that the area under $\triangle(ABC)$ is counted once with minus (for $\Trap(AB)$) and then with plus (for $\Trap(BC)$ and $\Trap(CA)$).	
	Right: splitting of a trapezoid $\Trap(AB)$ into a right triangle $\triangle(ABD)$ and a rectangle $\Trap(AD)$. 
}
\label{fig:triangle_and_trapezoid_splitting}
\end{figure}

Then 
\begin{equation}\label{eq:triangle_splitting}
o(ABC) \tilde\chi^{\mM}_{\triangle(ABC)}
=
-
o(BC) \tilde\chi^{\mM}_{\Trap(BC)}
+
o(AC) \tilde\chi^{\mM}_{\Trap(AC)}
-
o(AB) \tilde\chi^{\mM}_{\Trap(AB)}
,
\end{equation}
where $\Trap(XY)$ is a trapezoid with vertices $X$, $Y$, $X'$, and $Y'$, by $\bullet'$ we denote a projection on the $x$-axis (i.e., on $e_1$), and $o(XY)$ is the orientation of the two points $X$ and $Y$ on the $x$-axis defined as $\sgn(\overrightarrow{XY}\cdot e_1)$.

The formula \eqref{eq:triangle_splitting} is quite intuitive: indeed, as seen in Figure \ref{fig:triangle_splitting}, the area under the triangle will be counted with the minus sign when evaluating $-o(AB) \tilde\chi^{\mM}_{\Trap(AB)}$ and with the plus sign when evaluating $-o(BC) \tilde\chi^{\mM}_{\Trap(BC)} + o(AC) \tilde\chi^{\mM}_{\Trap(AC)}$.
The proof of \eqref{eq:triangle_splitting} is given in Appendix \ref{sec:triangle_splitting_proof}.

A trapezoid $\Trap(AB)$ can further be split into a right triangle and a rectangle (see an illustration in Figure \ref{fig:trapezoid_splitting}):
\[
o(AB) \tilde\chi^{\mM}_{\Trap(AB)} = o(ABD) \tilde\chi^{\mM}_{\triangle(ABD)} + o(AD)\tilde\chi^{\mM}_{\Trap(AD)}
,
\]
where $D := (B_1,A_2)$.
(Here indices $1$ and $2$ refer to the $x$- and $y$-coordinates of a point in $\bbR^2$.)

Thus, we reduced our problem to two problems: (1) computing $S'_{ABC}(\mM; c_1,c_2,c_3)$, where $AB$ and $BC$ are parallel to $x$ and $y$ axes respectively, and (2) computing
\[
\sum_{i,j\in\bbZ} o(AD) \tilde\chi^{\mM}_{\Trap(AD)}(i,j) (c_1 i+c_2 j+c_3)
,
\]
where $AD$ is parallel to the $x$-axis.
The latter can be computed analytically using the fact that the function $o(AD) \tilde\chi^{\mM}_{\Trap(AD)}(i,j)$ is symmetric with respect to rotation by the arc length $\pi$ around the center $O$ of the rectangle $\Trap(AD)$:
\begin{equation}\label{eq:sum_over_a_rectangle_answer}
\sum_{i,j\in\bbZ} o(AD) \tilde\chi^{\mM}_{\Trap(AD)}(i,j) (c_1 i+c_2 j+c_3)
=
(D_1-A_1) A_2\,(c_1 i+c_2 j+c_3)\big|_{(i,j)=O}
.
\end{equation}

\subsection{Computing the Sum over a Right Triangle}

It remains to develop an algorithm for computing $S'_{ABC}(\mM; c_1,c_2,c_3)$ for $AB$ and $BC$ parallel to the $x$- and the $y$-axis respectively, where $S'$ is defined by \eqref{eq:Sprime}.

\begin{figure}
\begin{center}
\hfill
\subfigure[]{\label{fig:right_triangle}\includegraphics{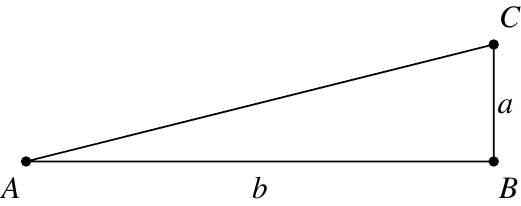}}
\hfill\hfill
\subfigure[]{\label{fig:right_triangle_shifted}\includegraphics{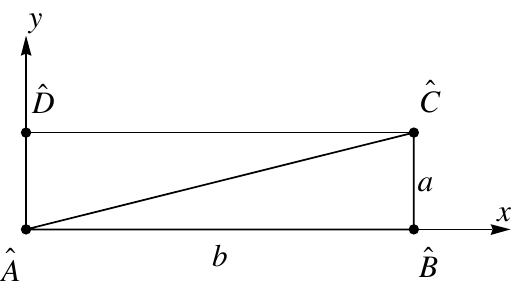}}
\hfill$\mathstrut$
\end{center}
\caption{A right triangle $\triangle(ABC)$ (left) and its shifted copy $\triangle(\hA\hB\hC)$ together with its rotated copy $\triangle(\hC\hD\hA)$ (right).
}
\end{figure}

Let $A=(A_1, A_2)$, $B=A+b e_1$, $C=B+a e_2$ (see Figure \ref{fig:right_triangle}).
We assume that both $a\in\bbZ$ and $b\in\bbZ$ are nonzero (otherwise $\triangle(ABC)$ is degenerate, and $S'_{ABC}(\mM; c_1,c_2,c_3)$ is zero).

We shift the points $A$, $B$, and $C$ so that $A$ coincides with the origin (see Figure \ref{fig:right_triangle_shifted}).
That is, we introduce the points $\hA = (0,0)$, $\hB = (b,0)$, $\hC = (b,a)$, and change the variables of summation $i\to i-A_1$, $j\to j-A_2$:
\begin{align*}
S'_{ABC}(\mM; c_1,c_2,c_3)
=~&
o(ABC)\sum_{i,j\in\bbZ} \tilde\chi^{\mM}_{\triangle(ABC)}(i+A_1,j+A_2)\,(c_1 i+c_2 j+c_4)
\\ =:~&
S'_{\hA\hB\hC}(\mM; c_1,c_2,c_4)
,
\end{align*}
where $c_4 = c_3 + c_1 A_1 + c_2 A_2$.

Second, notice that since $\triangle(\hA\hB\hC)$ and its copy rotated by $\pi$, $\triangle(\hC\hD\hA)$ (see an illustration in Figure \ref{fig:right_triangle_shifted}), together compose a rectangle, we can use \eqref{eq:sum_over_a_rectangle_answer} and express
\[
S'_{\hA\hB\hC}(\mM; 0,0,c_4)
= \smfrac12 (S'_{\hA\hB\hC}(\mM; 0,0,c_4) +S'_{\hC\hD\hA}(\mM; 0,0,c_4))
= \smfrac12 a b\, c_4.
\]
Thus, it remains to compute $S'_{\hA\hB\hC}(\mM; c_1,c_2,0)$.

Third, notice that we can reduce it to the case $a,b>0$ by doing reflections with respect to the axes (e.g., reflection around the $x$-axis corresponds to changing $a\to-a$, $c_1\to-c_1$, $o(\hA\hB\hC)\to-o(\hA\hB\hC)$).
Hence we assume that $a,b>0$ and therefore $o(ABC)=o(\hA\hB\hC)=1$.

Last, note that the function $\tilde\chi^{\mM}_{\triangle(\hA\hB\hC)}(i,j)$ can be described as follows:
\[
\tilde\chi^{\mM}_{\triangle(\hA\hB\hC)}(i,j)
=
\begin{cases}
1         & 0<i<b,~0<j<\smfrac ab i, \\
\smfrac12 & 0<i<b,~j=0, \\
\smfrac12 & 0<i<b,~j=\smfrac ab i, \\
\smfrac12 & i=b,~0<j<a, \\
\smfrac\alpha{2\pi} & i=0,~j=0, \\[0.2em]
\smfrac\beta{2\pi}  & i=b,~j=0, \\
\smfrac\gamma{2\pi} & i=b,~j=a, \\
0 & \text{otherwise,}
\end{cases}
\]
with
\begin{align*}
\alpha =~& \ang(\mM^{-1}(b e_1)\times \mM^{-1}e_3 \,,~ \mM^{-1}(b e_1+a e_2)\times \mM^{-1}e_3), \\
\beta  =~& \ang(\mM^{-1}(b e_1)\times \mM^{-1}e_3 \,,~ \mM^{-1}(a e_2)\times \mM^{-1}e_3), \\
\gamma =~& \ang(\mM^{-1}(a e_2)\times \mM^{-1}e_3 \,,~ \mM^{-1}(b e_1+a e_2)\times \mM^{-1}e_3)
,
\end{align*}
where $\ang(v,w):=\arccos(v\cdot w)$ for $u,v\in\bbR^3$.
	
Thus,
\begin{align*}
S'_{\hA\hB\hC}(\mM; c_1,c_2,0)
=~&
\sum_{i=1}^{b-1}
\sum_{j=1}^{\lfloor \frac{ai}b \rfloor} (c_1 i+c_2 j)
\\ ~&-
\smfrac12
\sum_{\substack{0<i<b \\ i\in\gcd(a,b)\bbZ}}
(c_1 i+c_2 \smfrac{ai}b)
+
\smfrac12 \sum_{j=1}^a (c_1 b+c_2 j)
+
\smfrac12 \sum_{i=1}^{b-1} (c_1 i+c_2 0)
\\ ~&+
\sum_{(i,j)\in\{\hA,\hB,\hC\}} \tilde\chi^{\mM}_{\triangle(\hA\hB\hC)}(i,j)
.
\end{align*}
Each of the terms in the second line is a sum of an arithmetic progression and can be expressed analytically.

Thus, it remains to compute
\[
\sum_{i=1}^{b-1}
\sum_{j=1}^{\lfloor \frac{ai}b \rfloor} (c_1 i+c_2 j)
=
\sum_{i=0}^{b-1}
\sum_{j=1}^{\lfloor \frac{ai}b \rfloor} (c_1 i+c_2 j).
\]
In what follows we will use the following two standard identities that can be easily proved by induction:
\begin{equation} \label{eq:sum_i}
\sum_{i=0}^{n-1} i = \frac{n (n-1)}2
,
\qquad
\sum_{i=0}^{n-1} i^2 = \frac{n (n-1) (2n-1)}6
.
\end{equation}
Using the second identity we can transform
\[
\sum_{i=0}^{b-1}
\sum_{j=1}^{\lfloor \frac{ai}b \rfloor} i
=
\sum_{i=0}^{b-1}
i\,\Big\lfloor \frac{ai}b \Big\rfloor
=
\sum_{i=0}^{b-1} i \Big(\frac{a i}{b} - \frac{1}{b} (a i \mod b) \Big)
= \frac{1}{6} (b-1) a (2 b-1) - S_{a,b},
\]
where
\begin{equation} \label{eq:Sab}
S_{a,b} := \sum_{i=0}^{b-1} \frac ib\,(a i \mod b)
,
\end{equation}
and
\begin{align*}
\sum_{i=0}^{b-1}
\sum_{j=1}^{\lfloor \frac{ai}b \rfloor} j
=~&
\sum_{i=0}^{b-1}
\frac12 \Big\lfloor \frac{ai}b \Big\rfloor \Big( \Big\lfloor \frac{ai}b \Big\rfloor+1 \Big)
\\ =~&
\frac12
\sum_{i=0}^{b-1}
\Big(\frac{a i}{b} - \frac{1}{b} (a i \mod b) \Big) \Big(1+\frac{a i}{b} - \frac{1}{b} (a i \mod b) \Big)
\\ =~&
\frac12
\sum_{i=0}^{b-1}
\frac{a i}{b} \Big(1+\frac{a i}{b}\Big)
-
\frac1{2b}
\sum_{i=0}^{b-1}
(a i \mod b)
-
\frac1b
\sum_{i=0}^{b-1}
\frac{a i}{b} (a i \mod b)
+
\frac1{2b^2}
\sum_{i=0}^{b-1}
(a i \mod b)^2
.
\end{align*}
The first, second, and fourth sums can be computed analytically as
\begin{align} \label{eq:sum_const_plus_i}
\sum_{i=0}^{b-1}
\frac{a i}{b} \Big(1+\frac{a i}{b}\Big)
=~&
\frac{a (b-1) (3b - a + 2ab)}{6 b},
\\ \label{eq:sum_ai_mod_b}
\sum_{i=0}^{b-1}
(a i \mod b)
=~&
\gcd(a,b)
\sum_{i=0}^{\smfrac{b}{\gcd(a,b)}-1}
\gcd(a,b)\,i
=
\frac{b (b-\gcd(a,b))}2,
\\ \label{eq:sum_ai_mod_b_square}
\sum_{i=0}^{b-1}
(a i \mod b)^2
=~&
\gcd(a,b)
\sum_{i=0}^{\smfrac{b}{\gcd(a,b)}-1}
(\gcd(a,b)\,i)^2
=
\frac{b (b-\gcd(a,b)) (2b-\gcd(a,b))}6,
\end{align}
and the third sum again reduces to \eqref{eq:Sab}:
\[
\frac1b
\sum_{i=0}^{b-1}
\frac{a i}{b} (a i \mod b)
= \frac ab S_{a,b}.
\]
Here the identity \eqref{eq:sum_const_plus_i} follows from the standard sums \eqref{eq:sum_i}, and to compute \eqref{eq:sum_ai_mod_b} and \eqref{eq:sum_ai_mod_b_square} we use (i) the fact that the numbers $(a i \mod b)$, $i=0,1,\ldots,b-1$, are essentially the numbers $\big\{0,\gcd(a,b),\ldots,\big(\smfrac{b}{\gcd(a,b)}-1\big)\gcd(a,b)\big\}$ repeated $\gcd(a,b)$ times each, and (ii) again the standard sums \eqref{eq:sum_i}.

\subsection{Computing The Sum $S_{a,b}$}\label{sec:computing_Len:Sab}

For computing the sum \eqref{eq:Sab} with $a,b\in\bbZ$ we propose a Euclidean-like algorithm with complexity $\calO(\log(a+b))$, which consists in iteratively reducing the problem with the parameters $(a,b)$ to the problem with the parameters $(b\mod a,a)$.

Using the following identity \cite[p.\thinspace 85]{GrahamKnuthPatashnik2004}:
\[
\lfloor a x \rfloor = \sum_{j=0}^{a-1} \left\lfloor x + \frac{j}{a} \right\rfloor
\quad \forall x\in\bbR
\]
with $x=i/b$ (we assume $b\ne0$, since $b=0$ is a trivial case), express
\[
\frac{1}{b}(a i \mod b)
=
\frac{a i}{b} - \left\lfloor \frac{a i}{b} \right\rfloor
= \frac{a i}{b} - \sum_{j=0}^{a-1} \left\lfloor \frac{i}{b} + \frac{j}{a} \right\rfloor
.
\]
Hence transform
\begin{displaymath}
S_{a,b}
=
\sum_{i=0}^{b-1} \frac{i}{b}\,(a i \mod b)
= \sum_{i=0}^{b-1} i\,\bigg(\frac{a i}{b} - \sum_{j=0}^{a-1} \left\lfloor \frac{i}{b} + \frac{j}{a} \right\rfloor\bigg)
= \frac{a}{b} \sum_{i=0}^{b-1} i^2 - \sum_{i=0}^{b-1} \sum_{j=0}^{a-1} i\left\lfloor \frac{i}{b} + \frac{j}{a} \right\rfloor
=: S_1 - S_2.
\end{displaymath}
The first sum is trivial: $S_1 = \frac{1}{6} (b-1) a (2 b-1)$.

For computing the second sum, notice that $\big\lfloor \frac{i}{b} + \frac{j}{a} \big\rfloor$ equals to $1$ if $\frac{i}{b} + \frac{j}{a}\ge 1$ and $0$ otherwise, and
\[
\frac{i}{b} + \frac{j}{a}\ge 1
~~\Leftrightarrow~~
i \ge b - \frac{b j}{a}
~~\Leftrightarrow~~
i \ge b - \left\lfloor \frac{b j}{a}\right\rfloor
.
\]
Hence
\begin{align*}
S_2
=~&
\sum_{i=0}^{b-1} \sum_{j=0}^{a-1} i\left\lfloor \frac{i}{b} + \frac{j}{a} \right\rfloor
=
\sum_{j=0}^{a-1} \sum_{i=b - \left\lfloor \frac{b j}{a}\right\rfloor}^{b-1}  i
=
\sum_{j=0}^{a-1} \frac{1}{2} \left\lfloor \frac{b j}{a}\right\rfloor \left(2 b-1 - \left\lfloor \frac{b j}{a}\right\rfloor \right)
\\ =~&
\frac{1}{2} \sum_{j=0}^{a-1} \left(\frac{b j}{a} - \frac{1}{a} (b j \mod a)\right) \left(2 b-1 - \frac{b j}{a} +  \frac{1}{a} (b j \mod a) \right)
\\ =~&
\frac{1}{2} \sum_{j=0}^{a-1} \left(
	\frac{b (2 b-1)}{a}\, j
	- \frac{b^2}{a^2}\,j^2
	- \frac{2b-1}{a} (b j \mod a)
	+ \frac{2 b}{a^2}\,j(b j \mod a)
	- \frac{1}{a^2} (b j \mod a)^2
\right).
\end{align*}
We next compute individual sums in $S_2$.

Using \eqref{eq:sum_ai_mod_b} and \eqref{eq:sum_ai_mod_b_square}, evaluate the individual sums in $S_2$:
\begin{align*}
S_2
=~&
\frac{1}{4} (a-1) b (2b-1)
-
\frac{(a-1) (2 a-1) b^2}{12 a}
\\ ~&
+
\frac{(1-2 b)}{2 a} \frac{(a-\gcd(a,b))\,a}{2}
+ \frac{b}{a} S_{b,a}
- \frac{(a-\gcd(a,b))\,(2 a-\gcd(a,b))}{12 a}
.
\end{align*}
Substituting this back to $S_{a,b}$ and collecting the terms yields
\[
S_{a,b}
= 
\frac{3a^2 b +3 a b^2+a^2-3ab +b^2 - 6ab \gcd(a,b) +\gcd(a,b)^2}{12 a} - \frac{b}{a} S_{b,a}
.
\]
Finally, notice that $S_{b,a} = S_{b\mod a, a}$.
We thus manage to reduce the problem from
$(a,b)$ to $(b\mod a,a)$, which yields the $\calO(\log(a+b))$ algorithm.

\subsection{Complexity of the Algorithm}\label{sec:computing_Len:complexity}

\begin{proposition}\label{prop:complexity_of_len}
The number of operations of the algorithm described in Sections \ref{sec:computing_Len:reductions}--\ref{sec:computing_Len:Sab} is at most $\calO\big(\log({\rm diam}(T))+\log(|r|)\big)$.
\end{proposition}
\begin{proof}
Evidently, $|r|$ is not increased as a result of reducing to $\gcd(r_1,r_2,r_3)=1$ (Section \ref{sec:computing_Len:reductions:gcd}).

In reducing to $r=e_3$ (Section \ref{sec:computing_Len:reductions:r_e3}), we can choose $0\leq c_{12}<|r_2|$, $0\leq d_{12}<|r_1|$, $0\leq c_3<|r_3|$, $0\leq d_3<\gcd(|r_1|,|r_2|)$.
Then, up to a constant factor, $\|\mM\|$ can be estimated as
\[
\|\mM\|_{\infty} \lesssim \max(|r_3| + \gcd(|r_1|,|r_2|), 1, c_3+d_3) \max\big(c_{12}+d_{12}, \smfrac{|r_1|+|r_2|}{\gcd(|r_1|,|r_2|)}, 1\big)
\lesssim |r|^2.
\]
And hence $\diam(\mM T) \lesssim r^2 \diam(T)$.

The triangles considered in Section \ref{sec:computing_Len:reductions:triangle} are the bases of the truncated prisms of Section \ref{sec:computing_Len:reductions:prism}, which are projections of faces of $\mM T$, and hence the diameter of each $\triangle$ in Section \ref{sec:computing_Len:reductions:triangle} is at most $\calO(r^2 \diam(T))$.
Similarly one can deduce that both $a,b$ in each sum $S_{a,b}$ are of the order $\calO(r^2 \diam(T))$ and hence the Euclidean-like algorithm of Section \ref{sec:computing_Len:Sab} takes $\calO\big(\log(r^2 \diam(T))\big) = \calO\big(\log(\diam(T))+\log|r|\big)$ operations.
\end{proof}

\section{Numerical Tests}\label{sec:num}

Numerical tests were conducted in order to numerically study the accuracy and stability of the method.

The atoms interact with the Lennard--Jones potential $\phi(z) = -2 |z|^{-6}+|z|^{-12}$ under which the FCC lattice is stable.
The cut-off radius is chosen to be $3.2$.

The lattice vectors are chosen as $a_1=(0,\smfrac1{\sqrt2},\smfrac1{\sqrt2})$, $a_2=(\smfrac1{\sqrt2},0,\smfrac1{\sqrt2})$, $a_3=(\smfrac1{\sqrt2},\smfrac1{\sqrt2},0)$.
The reference lattice $\calL\setminus\calL_\D$ consists of a crystal whose atoms formed a cube with the side $2\sqrt{2} N$ ($N=8,16$).
The Dirichlet-type boundary conditions are imposed by introducing additional atoms with fixed positions, $\calL_\D$.
A single atom at the origin is removed thus forming a vacancy defect.
In total, for $N=16$, the number of (unconstrained) atoms in the atomistic system is $\#(\calL\setminus\calL_\D) = 125022$.

A macroscopic uniform deformation with
\[
\mF = \scriptsize \begin{pmatrix}1&0.01&0.02\\ 0&1&0.015\\ 0&0&1\end{pmatrix}
\]
is applied to the constrained atoms $\calL_\D$.

\begin{figure}
\begin{center}
\includegraphics[width=10cm]{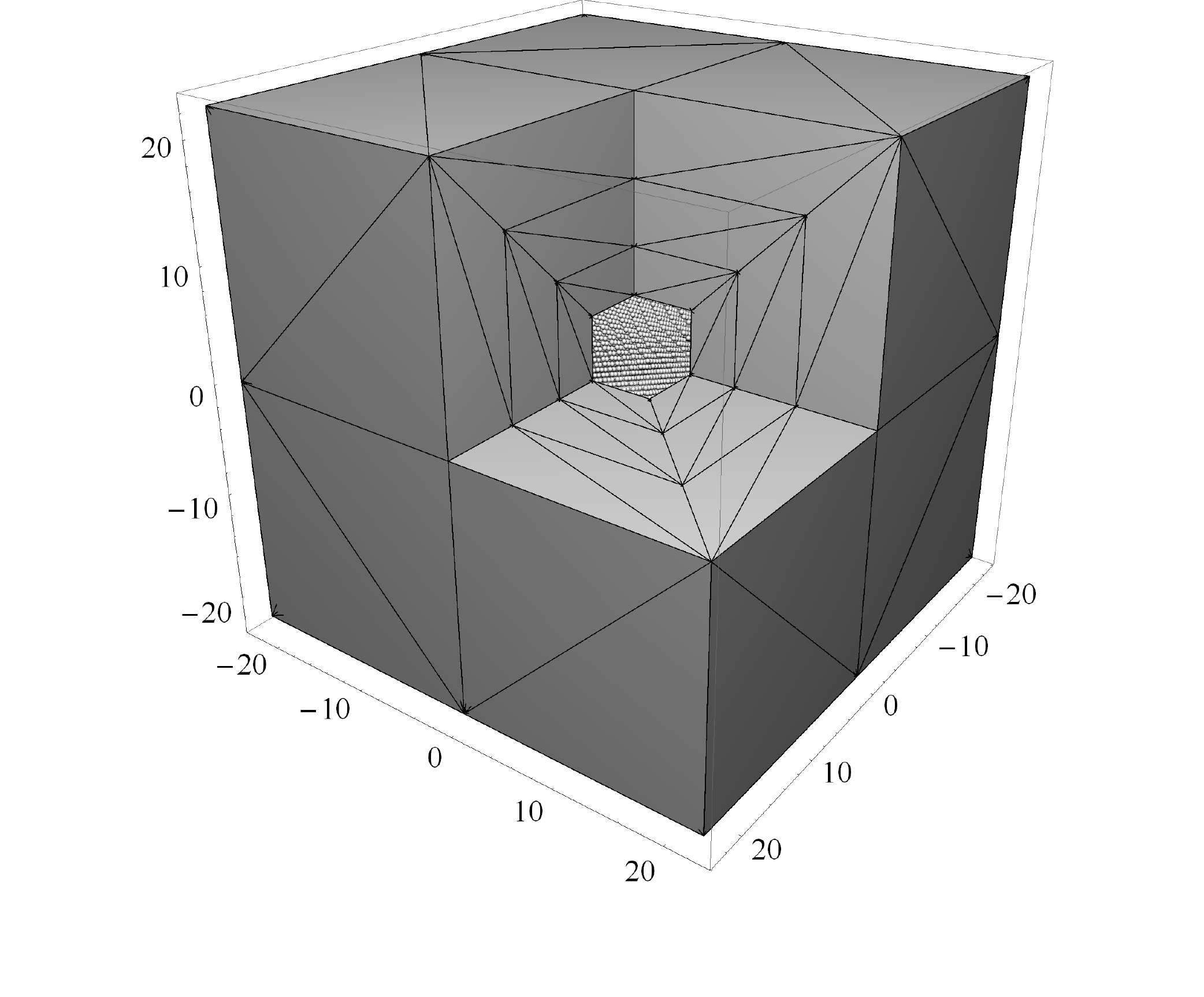}
\end{center}
\caption{Geometry of the A/C coupling: a triangulated continuum region and atoms in the atomistic region for $K=2$ and $N=16$.}
\label{fig:3d-geom}
\end{figure}

The atomistic region $\Omega_\a$ is chosen as a smaller cube with the side $2\sqrt{2} K$, $K=2,3,\ldots,11$ (see an illustration in Figure \ref{fig:3d-geom}).
A quasiradial mesh with mesh size $h=1$ near the A/C interface is chosen in accordance with the optimal choice of meshes in 2D \cite{OrtnerShapeev:qce_analysis:preprint}.
(Note that $h=1$ corresponds to a fully refined mesh near the interface.)

\subsection{Accuracy Test}

In the numerical tests, the exact and the approximate solutions were computed using Newton's method of solving the equilibrium equations, with the initial guess being an undeformed configuration.

The results of the computations are shown in Figure \ref{fig:error}.
The difference in the $\W^{1,\infty}$-seminorm between the approximate and the exact deformation is plotted on the left, and the difference between the energies $|E^h(u^h)-E(u)|$ is plotted on the right.
These errors are plotted against the number of DoF in the system.
One can see that the error in the $\W^{1,\infty}$-seminorm converges with the rate of at least $\calO\big({\rm DoF}^{-1}\big)$ and the error in energy is close to $\calO\big({\rm DoF}^{-5/3}\big)$.

\begin{figure}
\begin{center}
\hfill
\includegraphics{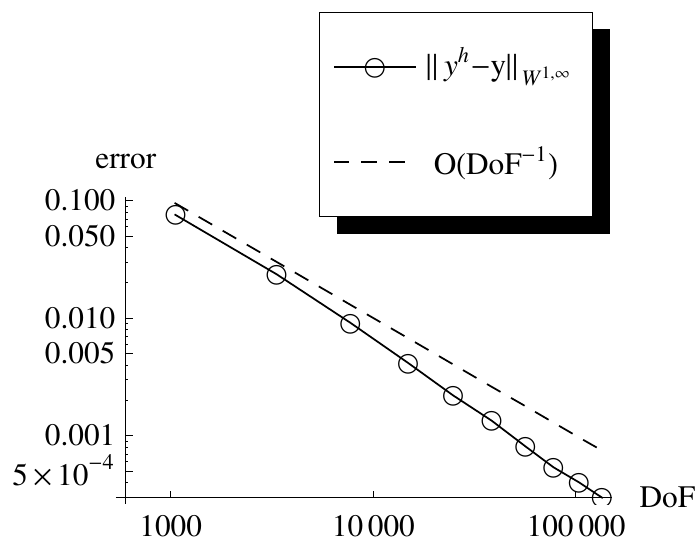}
\hfill\hfill
\includegraphics{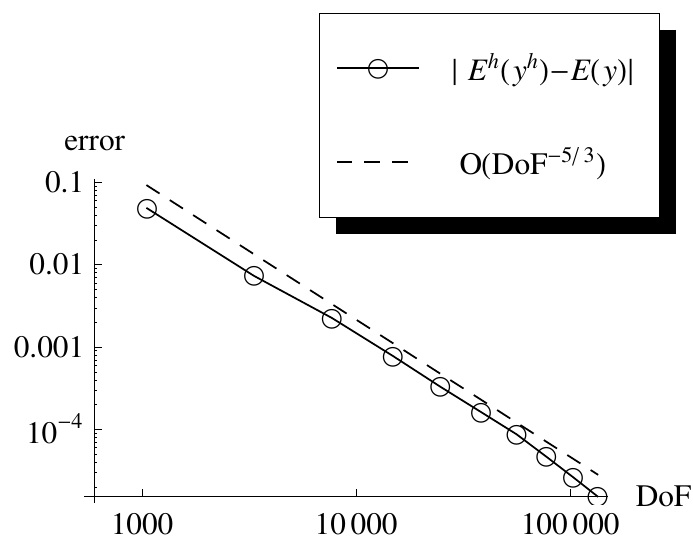}
\hfill$\mathstrut$
\end{center}
\caption{Results of computations. The $\W^{1,\infty}$-error (left) and error in energy (right) are plotted against the number of degrees of freedom DoF.
	The error in the $\W^{1,\infty}$-seminorm converges with the rate of at least $\calO\big({\rm DoF}^{-1}\big)$, and the error in energy is close to $\calO\big({\rm DoF}^{-5/3}\big)$.
}
\label{fig:error}
\end{figure}

\subsection{Stability Tests for a Bravais Lattice}

We also conducted stability tests for a Bravais lattice (i.e., with no defects) to verify that the stability region of the proposed coupling $E^h$ is not smaller than the stability region of the atomistic model $E$.
This is crucial in numerically studying defects: one must ensure that the onset of instability occurs due to motion of a defect but not due to artifacts of the coupling.

We take a Bravais lattice $\calL\setminus\calL_\D$ similar to the one previously described but with no removed atoms.
The macroscopic uniform deformation gradient
\[
\mF =
\begin{pmatrix}1+t&0.05&0.02\\ 0&1+s&0.01\\ 0&0&1\end{pmatrix}
\]
is applied to the constrained atoms $\calL_\D$.
The two parameters, $t$ and $s$, were varied in the range between $-0.2$ and $0.2$.

We computed stability regions for the coupling $E^h$ and compared it to the stability region of the exact atomistic model on an infinite lattice.
A stability region is defined as the set of parameters $(t,s)$ for which the model is stable.
Stability of $E^h$ was determined by numerically testing whether the Hessian $\ddel E^h$ is positive definite.
The stability of the atomistic model was computed analytically with the help of the Fourier transform \cite{HudsonOrtner}.

\begin{figure}
\begin{center}
\hfill
\includegraphics{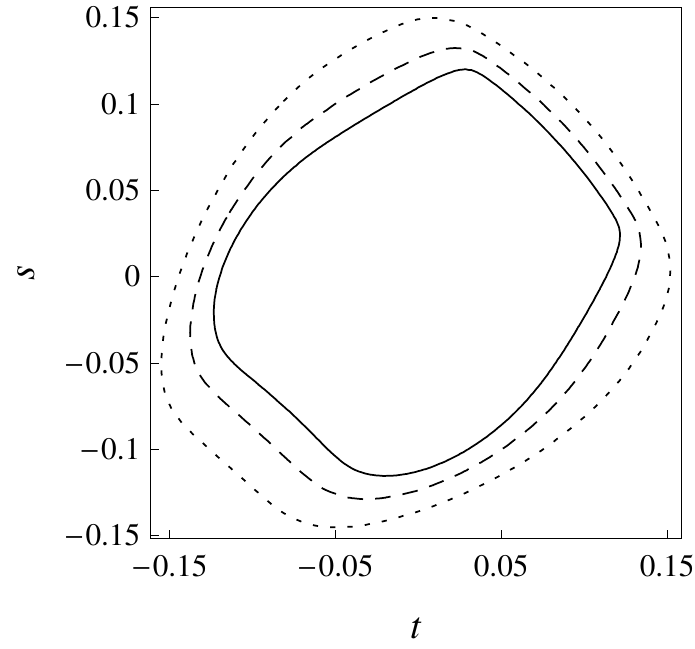}
\hfill$\mathstrut$
\end{center}
\caption{Stability regions for a Bravais lattice.
	The solid line corresponds to the exact stability region of the infinite lattice, the dashed line corresponds to the coupling with $N=16$ and $K=8$, and the dotted line corresponds to the coupling with $N=8$ and $K=4$.
	The results suggest that the coupled A/C system does not lose stability earlier than the original atomistic system.}
\label{fig:stability_regions}
\end{figure}

The stability regions are plotted in Figure \ref{fig:stability_regions}.
The solid line corresponds to the exact stability region of the infinite lattice, the dashed line corresponds to the coupling with $N=16$ and $K=8$, and the dotted line corresponds to the coupling with $N=8$ and $K=4$.
One can see that stability regions of the coupling strictly contain the exact stability region and seem to approach it as $N$ and $K$ increase, which is the desired behavior of the A/C coupling.

\section{Discussion and Conclusion}\label{sec:conclusion}
In the present paper the consistent A/C coupling \cite{Shapeev2011} has been extended to 3D for two-body potentials.
The proposed method couples the atomistic equations with the modified Cauchy--Born continuum model.
The continuum energy of the modified model can be evaluated efficiently, as discussed in Section \ref{sec:computing_Len}.
Although the stability of such a modified continuum model has not been studied analytically in the existing literature, the numerical tests suggest that the proposed coupling is stable whenever the atomistic model is stable.
The numerical tests also confirm convergence of the proposed coupling to the exact solution.

The major challenge yet to be solved is an extension of the present method to many-body interaction.
A one-dimensional consistent coupling for such an interaction exists \cite{LiLuskin:eam.qc}; however, it does not seem obvious how to define continuum approximations to many-body bond energies in many dimensions.

\section*{Acknowledgments}
The author thanks Christoph Ortner for his valuable comments and advice that led to improvement of the manuscript.
The author also appreciated a discussion with Alla Merzakreeva who brought the author's attention to the existing developments in computer science related to integer sums.

\appendix

\section{Splitting a Simplex into Truncated Prisms} \label{sec:triangle_splitting_proof}

\subsection{Auxiliary Definitions}

We define an orientation of points $X^{(1)}$, \ldots, $X^{(d+1)} \in \bbR^d$ as an orientation of the basis $\big(X^{(2)}-X^{(1)},\ldots,X^{(d+1)}-X^{(1)}\big)$; i.e.
\begin{equation}\label{eq:orientation_general_def}
o_{d}(X^{(1)}, \ldots, X^{(d+1)})
:=
\sgn\circ\det
\begin{pmatrix}
X^{(2)}_1 - X^{(1)}_1 & X^{(2)}_2 - X^{(1)}_2 & \ldots & X^{(2)}_{d} - X^{(1)}_{d} \\ 
X^{(3)}_1 - X^{(1)}_1 & X^{(3)}_2 - X^{(1)}_2 & \ldots & X^{(3)}_{d} - X^{(1)}_{d} \\ 
\vdots & \vdots & \ddots & \vdots \\
X^{(d+1)}_1 - X^{(1)}_1 & X^{(d+1)}_2 - X^{(1)}_2 & \ldots & X^{(d+1)}_{d} - X^{(1)}_{d} \\ 
\end{pmatrix}
,
\end{equation}
$o_{d}(X^{(1)}, \ldots, X^{(d+1)}) \in \{-1,0,1\}$.

Next, we define the characteristic function of a truncated prism $P = \P(X^{(1)}, \ldots, X^{(d)})$ for $d$ points $X^{(1)}$, \ldots, $X^{(d)} \in \bbR^d$.
If the plane through these $d$ points is perpendicular to the hyperplane $x_d=0$, then we let $\cchi_P := 0$.
Otherwise, let $x_d = \sum_{i=1}^{d-1} \alpha_i x_i$ be an equation of such a plane and define $\tilde{\cchi}_P\in \L_1(\bbR^d)$ as
\begin{equation} \label{eq:chi-P-tilde}
\tilde{\cchi}_P
:=
\begin{cases}
1, & 0<x_d<\sum_{i=1}^{d-1} \alpha_i x_i \text{ and } \xi\in{\rm conv}((X^{(1)})', \ldots, (X^{(d)})'), \\ 
-1, & 0>x_d>\sum_{i=1}^{d-1} \alpha_i x_i \text{ and } \xi\in{\rm conv}((X^{(1)})', \ldots, (X^{(d)})'), \\ 
0 & \text{otherwise},
\end{cases}
\end{equation}
and
\begin{equation}\label{eq:chi-P}
\cchi_P(x) := \lim_{\rho\to0} \frac1{|B_\rho(x)|} \int_{B_\rho(x)} \tilde{\cchi}_P(\xi)\dd\xi
,
\end{equation}
where by $\bullet'$ we denote an orthogonal projection on the hyperplane $x_d=0$.

Note that $\cchi_P = \tilde{\cchi}_P$ almost everywhere ($\cchi_P$ is defined pointwise in $\bbR^d$ and may take intermediate values conforming with the definition \eqref{eq:chi_def}).
Also note that $\cchi_P$ is a characteristic function of a polytope $P$ in the sense of \eqref{eq:chi_def} only when the face ${\rm conv}(X^{(1)}, \ldots, X^{(d)})$ is ``high enough'' (i.e., has a large enough $x_d$-coordinate).
Otherwise, writing $P = \P(X^{(1)}, \ldots, X^{(d)})$ is formal and does not refer to a proper polytope in $\bbR^d$.

\subsection{Formulation of the Result}

Let $A^{(1)}$, \ldots, $A^{(d+1)} \in \bbR^d$ so that $o_d(A^{(1)}$, \ldots, $A^{(d+1)})=1$.
Denote
\[
\begin{array}{rr@{}l}
\text{the simplex~~}&
T :=~& {\rm conv}(A^{(1)}, \ldots, A^{(d+1)}),
\\
\text{the truncated prisms~~}&
P^{(k)} :=~&
\P(A^{(1)}, \ldots, A^{(k-1)}, A^{(k+1)}, \ldots, A^{(d+1)})
,
\\
\text{their orientations~~}&
o^{(k)} :=~&
o_{d-1}((A^{(1)})', \ldots, (A^{(k-1)})', (A^{(k+1)})', \ldots, (A^{(d+1)})')
,
\\
\text{and the faces of $T$,~~}&
F^{(k)} =~& {\rm conv}(A^{(1)}, \ldots, A^{(k-1)}, A^{(k+1)}, \ldots, A^{(d+1)})
,
\end{array}
\]
where $k=1,\ldots,d+1$.
Here we identify points on the hyperplane $x_d=0$ with points in $\bbR^{d-1}$ when computing orientations $o^{(k)}$ via \eqref{eq:orientation_general_def}

We prove the following proposition.
\begin{proposition}\label{prop:splitting}
\begin{equation}\label{eq:T-through-P}
\cchi_T = -\sum_{k=1}^{d+1} (-1)^k
	o^{(k)}\,
	\cchi_{P^{(k)}}
\end{equation}
pointwise in $\bbR^d$, where $\cchi_T$ is defined by \eqref{eq:chi_def}.
\end{proposition}

\subsection{Proof}

In the following two lemmas we prove that \eqref{eq:T-through-P} holds almost everywhere.
\begin{lemma}\label{lem:tr-splitting-lemma-one}
The relation \eqref{eq:T-through-P} holds almost everywhere in $\bbR^d$ if $A^{(k)}_d\geq 0$ for all $k=1,\ldots,d+1$ (i.e., if $T$ is entirely above the hyperplane $x_d=0$).
\end{lemma}
\begin{proof}
Since all vertices $A^{(k)}$ lie above the hyperplane $x_d=0$, the function $\cchi_P$ defined by \eqref{eq:chi-P-tilde} and \eqref{eq:chi-P} equals $1$ or $0$ almost everywhere and hence is the characteristic function of a proper truncated prism $P$.

Fix an arbitrary test function $f\in \C^{\infty}(\bbR^d)$, and let
\[
g(x_1,\ldots,x_{d-1},x_d) := e_d \int_0^{x_d} f(x_1,\ldots,x_{d-1},\xi) \dd\xi
\]
so that $f = {\rm div} g$
(recall that $e_1,\ldots,e_d$ is the canonical basis of $\bbR^d$).
Multiply \eqref{eq:T-through-P} by $f$ and apply the divergence theorem:
\begin{equation}\label{eq:T-through-P-intermediate}
\int_{\partial T} g\cdot n^T \dd\gamma
=
- \sum_{k=1}^{d+1} (-1)^k \int_{\partial P^{(k)}} o^{(k)} g\cdot n^{P^{(k)}}\dd\gamma,
\end{equation}
where $n^\bullet$ is the outward normal vector (to $T$ or $P^{(k)}$ respectively).

Next, note that $g=0$ on the base of the truncated prism (i.e., for $x_d=0$) by the definition of $g$ and $g\cdot n^{P^{(k)}}=0$ on the sides of the truncated prism (i.e., below $F^{(k)}$).
Hence, in both parts of relation \eqref{eq:T-through-P-intermediate} we have the sum over faces $F^{(k)}$ of the integrals of $g\cdot n^\bullet$, and it remains only to verify that $n^T= -(-1)^k o^{(k)} n^{P^{(k)}}$ for all $k$.

To prove this, fix $k\in\{1,\ldots,d+1\}$ such that $P^{(k)}$ is not degenerate (otherwise the statement is trivial since $n^{P^{(k)}}=0$ and $n^T=0$ on that face of $T$).
Note that $n^{P^{(k)}}_d>0$ (i.e., the vector $n_{P^{(k)}}$ points upwards).
The agreement of the orientation of $n^T$ and of $o^{(k)}$ follows from the following chain of statements:
\begin{align*}
&
	o^{(k)}=(-1)^k
\\\Leftrightarrow~&
	o_d(A^{(1)}, \ldots, A^{(k-1)}, A^{(k+1)}, \ldots, A^{(d+1)},A^{(\ell)}+e_d)=(-1)^k
\\\Leftrightarrow~&
	o_d(A^{(1)}, \ldots, A^{(k-1)}, A^{(\ell)}+e_d, A^{(k+1)}, \ldots, A^{(d+1)})=1
\\\Leftrightarrow~&
	\text{$e_d$ is an inward vector w.r.t.\ $F^{(k)}$}
\\\Leftrightarrow~&
	n^T_d<0,
\end{align*}
where $\ell$ is any integer between $1$ and $d+1$ different from $k$.
Here in the first step we used the expansion of determinant by minors and in the third step the following fact:
since the basis $\big(A^{(2)}-A^{(1)}, \ldots, A^{(d+1)}-A^{(1)}\big)$ is positively oriented, a vector $v$ is an inward (or, resp., outward) vector w.r.t.\ $F^{(k)}$ if and only if the orientation of the basis $\big(A^{(2)}-A^{(1)}, \ldots, A^{(k-1)}-A^{(1)}, v, A^{(k+1)}-A^{(1)}, \ldots, A^{(d+1)}-A^{(1)}\big)$ is positive (or, resp., negative).
\end{proof}

\begin{lemma}\label{lem:tr-splitting-lemma-two}
The relation \eqref{eq:T-through-P} holds almost everywhere in $\bbR^d$.
\end{lemma}
\begin{proof}
In view of the previous lemma, we need only to show that both sides of \eqref{eq:T-through-P} are invariant w.r.t.\ $S_D$, a dilatation in $x_d$ by distance $D$.
More precisely, we need to note that $S_D \cchi_T = \cchi_{S_D T}$  almost everywhere (which follows directly from the definition of $\cchi$) and prove that
\begin{equation}\label{eq:chi_P-invar}
\sum_{k=1}^{d+1} (-1)^k
	o^{(k)}\,
	S_D\circ \cchi_{P^{(k)}}
=
\sum_{k=1}^{d+1} (-1)^k
	o^{(k)}\,
	\cchi_{S_D P^{(k)}}
.
\end{equation}

The proof of \eqref{eq:chi_P-invar} is based on noting that
\[
\cchi_{S_D \P(X^{(1)},\ldots,X^{(d)})}
=
S_D\circ \cchi_{\P(X^{(1)},\ldots,X^{(d)})} + \cchi_{\P(De_d+(X^{(1)})',\ldots,(De_d+X^{(d)})')}
,
\]
i.e., that $\chi$ of a shifted truncated prism is equal to shifted $\chi$ of a truncated prism plus $\chi$ of a prism of height $D$ with the base formed by projections $(X^{(1)})',\ldots,(X^{(d)})'$.
This can be verified by fixing $x\in{\rm conv}((X^{(1)})', \ldots, (X^{(d)})')$ (for $x$ elsewhere the statement is trivial) and considering three cases: $\sum_{i=1}^{d-1} \alpha_i x_i$ is (i) less than $\min(0,D)$, (ii) between $\min(0,D)$ and $\max(0,D)$, and (iii) greater than $\max(0,D)$.

We now take the difference between the left-hand side and the right-hand side of \eqref{eq:chi_P-invar}:
\begin{equation}\label{eq:chi-P-shifted-difference}
\sum_{k=1}^{d+1} (-1)^k
	o^{(k)}\,
	S_D\circ \cchi_{P^{(k)}}
-
\sum_{k=1}^{d+1} (-1)^k
	o^{(k)}\,
	\cchi_{S_D P^{(k)}}
=
\sum_{k=1}^{d+1} (-1)^k
	o^{(k)}\,
	\cchi_{S_D (F^{(k)})'}
\end{equation}
and notice that this difference depends only on projections of faces, $(F^{(k)})'$, but not on the $d$th coordinate of the points $A^{(1)},\ldots,A^{(d+1)}$.
Hence we can again shift these points so that Lemma \ref{lem:tr-splitting-lemma-one} applies to both $T$ and $S_D T$ and hence conclude that the difference \eqref{eq:chi-P-shifted-difference} equals $S_D\circ \cchi_T - \cchi_{S_D T} = 0$ almost everywhere.
\end{proof}

We now finalize the proof of Proposition \ref{prop:splitting}.
\begin{proof}[Proof of Proposition \ref{prop:splitting}]
Denote $f(x) = \cchi_T(x) + \sum_{k=1}^{d+1} (-1)^k
	o^{(k)}\,
	\cchi_{P^{(k)}}(x)$.
From \eqref{eq:chi_def} and \eqref{eq:chi-P} we have
\begin{equation} \label{eq:prop-proof-final}
f(x)
=
\lim_{\rho\to0} \frac1{|B_\rho(x)|} \int_{B_\rho(x)}
f(\xi) \dd\xi.
\end{equation}
Due to Lemma \ref{lem:tr-splitting-lemma-two}, $f(\xi)=0$ for almost all $\xi$; hence the right-hand side of \eqref{eq:prop-proof-final} is zero, hence $f(x)=0$ for all $x\in\bbR^d$.
\end{proof}

\section{A Matlab implementation of the computation of $\len{T,r}$} \label{sec:matlab}

The code given in this appendix closely follows the algorithm outlined in Section \ref{sec:computing_Len} except for the following optimization.

Instead of the splitting \eqref{eq:triangle_splitting} we introduce three additional points $D=(B_1, A_2)$, $E=(C_1, A_2)$, $F=(C_1, B_2)$ (see Figure \ref{fig:triangle_splitting_alt}) and represent
\begin{align*}
o(ABC) \tilde\chi^{\mM}_{\triangle(ABC)}
=
-~& o(ADB) \tilde\chi^{\mM}_{\triangle(ADB)}
\\ +~& o(AEC) \tilde\chi^{\mM}_{\triangle(AEC)}
\\ +~& o(FBC) \tilde\chi^{\mM}_{\triangle(FBC)}
\\ +~& o(EDB) \tilde\chi^{\mM}_{\square(EDBF)}
,
\end{align*}
where $\square(EDBF)$ is the rectangle $EDBF$.
The contribution of the latter is computed by a formula analogous to \eqref{eq:sum_over_a_rectangle_answer}.

\begin{figure}
\begin{center}
\includegraphics{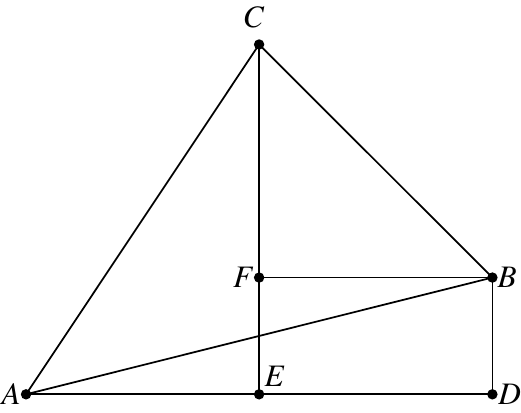}
\end{center}
\caption{An alternative splitting of $\triangle(ABC)$ into three right triangles and a rectangle.
}
\label{fig:triangle_splitting_alt}
\end{figure}

\subsection*{The Matlab Code}

{\small
\begin{verbatim}
function BondVol = BondVol_tetrahedron(v, r)
% BondVol(v,r) of the tetrahedron with vertices v

    % change the coordinates so that r = (0,0,1)
    [gcd12, c12, d12] = gcd(r(1),r(2));
    [gcd3, c3, d3] = gcd(gcd12,r(3));
    if(gcd3~=1)
        % reduce to the case gcd3=1
        r = r/gcd3;
        gcd12 = gcd12 /gcd3;
    end
    M = eye(3);
    if(r(1) ~=0 || r(2) ~=0 || r(3) == 0)
        M = [c12,d12,0;r(2)/gcd12, -r(1)/gcd12, 0;0,0,1];
        M = [-r(3),0,gcd12;0,1,0;c3,0,d3]*M;

        r = M*r;
        v = M*v;
    end

    invM = inv(M);
    BondVol = BondVol_prism(v(:,[1;2;3]), invM) + ...
        BondVol_prism(v(:,[4;3;2]), invM) + ...
        BondVol_prism(v(:,[4;2;1]), invM) + ...
        BondVol_prism(v(:,[1;3;4]), invM);
end

function BondVol = BondVol_prism(v, invM)
% BondVol(v,r) of the truncated prism,
% formed by three vertices v and their projections on the XY plane;
% r is assumed to be (0,0,1);
% invM*r and invM*v are the original positions.

    % shift the triangle (and the prism) so that v1 = (0,0,*)
    v([1 2],2)=v([1 2],2)-v([1 2],1);
    v([1 2],3)=v([1 2],3)-v([1 2],1);
    v([1 2],1)=[0;0];

    if(round(det([[1;1;1], v([1 2],:)']))==0)
        % if the prism is degenerate
        BondVol = 0;
    else
        % find coefficients [c4; c1; c2] of the function c1*i + c2*j + c4
        % that we sum
        c4c1c2 = [[1;1;1], v([1 2],:)']\v(3,:)';

        % reduce to integration over a triangle formed by (0,0), (v2x, v2y),
        % (v3x, v3y), which is further reduced to integration over three
        % right triangles and one rectrangle
        BondVol = + right_triangle_sum(c4c1c2, v([1 2],2), invM) ...
            - right_triangle_sum(c4c1c2, v([1 2],3), invM) ...
            - right_triangle_sum(c4c1c2 + ...
                [1;0;0]*([0 v([1 2],3)']*c4c1c2), ...
                v([1 2],2)-v([1 2],3), invM) ...
            + prod(v([1 2],3)-[v(1,2);0])*...
                ([1 0.5*(v([1 2],3)+[v(1,2);0])'] * c4c1c2);
    end
end

function ans = right_triangle_sum(c4c1c2, pt, invM)
% sum  c4 + c1 x + c2 y  over a right triangle (0,0), (pt(1),0), pt

    b = pt(1); % x-side
    a = pt(2); % y-side

    if(a==0 || b==0)
        ans = 0; return;
    end

    % reduce to a>0 and b>0
    orientation = sign(a)*sign(b);
    c4c1c2 = c4c1c2 .* [1;sign(b);sign(a)];
    invM = invM*diag([sign(b), sign(a), 1]);
    a=abs(a); b=abs(b);

    % sum the constant term
    ans = 1/2*a*b*c4c1c2(1);
    c1c2 = c4c1c2(2:3);

    % sum the linear terms, using reduction to Sab
    gcdab = gcd(a,b);
    Sab_ans=Sab(a,b,gcdab);
    ans = ans + c4c1c2(2) * (1/6*a*(b-1)*(2*b-1) - Sab_ans);
    ans = ans + c4c1c2(3) * (1/4*(a-1)*(b-1) + 1/4*(gcdab-1) ...
        + 1/12*(a^2)/b*(b-1)*(2*b-1) + 1/12/b*(b-gcdab)*(2*b-gcdab) ...
        - a/b*Sab_ans);

    % sum over sides:
    ans = ans + 0.5*(a-1)* [b a/2]*c1c2;
    ans = ans + 0.5*(b-1)* [b/2 0]*c1c2;

    % sum over the hypotenuse (we subtract half the contribution)
    ans = ans - 0.5*(gcdab-1)* [b/2 a/2]*c1c2;

    % sum over vertices
    v1 = cross(invM(:,1),invM(:,3)); % invM(:,1) == invM*[1;0;0]
    v2 = cross(invM(:,2),invM(:,3));
    v3 = -cross(invM*[b;a;0],invM(:,3));
    i = 0; j = 0;
    ans = ans + acos(dot(v1,-v3)/norm(v1)/norm(v3))/(2*pi) * [i j]*c1c2;
    i = b; j = a;
    ans = ans + acos(dot(-v2,v3)/norm(v2)/norm(v3))/(2*pi) * [i j]*c1c2;
    i = b; j = 0;
    ans = ans + acos(dot(-v1,v2)/norm(v1)/norm(v2))/(2*pi) * [i j]*c1c2;

    ans = real(ans) * orientation;
end

function ans = Sab(a,b, gcdab)
% sums i/b (a*i mod b) over i=0..b-1

    if(b==0)
        ans=0; return;
    end
    
    ans = 0; multiplier = 1;

    while(a ~= 0)
        ans = ans + multiplier * ...
            (3*b*a^2 + 3*b^2*a + a^2 - 3*a*b + b^2 - 6*a*b*gcdab ...
                + gcdab^2)/(12*a);
        multiplier = multiplier * (-b/a);
        old_b=b; b=a; a=mod(old_b,a);
    end
end

\end{verbatim}
}

\bibliographystyle{plain}
\bibliography{paper}

\end{document}